\documentclass[12pt]{amsart}
\usepackage{amscd}
\usepackage{amsmath}
\usepackage{latexsym}
\usepackage{graphicx}
\usepackage{amsfonts}
\usepackage{amssymb}
\usepackage[all]{xy}
\renewcommand{\atop}[2]{\genfrac{}{}{0pt}{}{#1}{#2}}

\newcommand{\edit}[1]{}

\newcommand{\rank}{\mathop{\mathrm{rank}}}
\newcommand{\disc}{\mathop{\mathrm{disc}}}
\newcommand{\oo}{\otimes}
\newcommand{\op}{\mathrm{op}}

\newcommand{\Cl}{C}
\newcommand{\Clone}{C_{\langle1\rangle}}

\newcommand{\bP}{\mathbb{P}}
\newcommand{\R}{\mathbb{R}}
\newcommand{\C}{\mathbb{C}}
\renewcommand{\H}{\mathbb{H}}
\newcommand{\bQ}{\mathbb{Q}}
\newcommand{\Z}{\mathbb{Z}}
\newcommand{\Q}{\mathbf{Q}}
\newcommand{\GW}{\textrm{GW}}
\newcommand{\XG}{X_G}   
\newcommand{\Sq}{\mathrm{Sq}}
\newcommand{\cA}{\mathcal{A}}
\newcommand{\cB}{\mathcal{B}}

\newcommand{\cE}{\mathcal{E}}
\newcommand{\cH}{\mathcal{H}}
\newcommand{\cL}{\mathcal{L}}
\newcommand{\cO}{\mathcal{O}}
\newcommand{\pt}{{\mathrm{pt}}}

\newcommand{\an}{\mathrm{an}}
\newcommand{\et}{\mathrm{et}}

\newcommand{\tors}{\mathrm{tors\,}}
\renewcommand{\top}{{\mathrm{top}}}
\newcommand{\bara}{\bar{a}}

\newcommand{\tH}{\widetilde{H}}

\numberwithin{equation}{section}
\theoremstyle{plain}
\newtheorem{theorem}[equation]{Theorem}
\newtheorem{maintheorem}[equation]{Main Theorem}

\newtheorem{corollary}[equation]{Corollary}
\newtheorem{proposition}[equation]{Proposition}
\newtheorem{lemma}[equation]{Lemma}
\newtheorem{substuff}{\bf Remark}[equation] 

\newtheorem{subcor}[substuff]{Corollary}

\theoremstyle{definition}

\newtheorem{defn}[equation]{Definition}

\newtheorem{example}[equation]{Example}
\newtheorem{subex}[substuff]{Example}

\theoremstyle{remark}
\newtheorem{remark}[equation]{Remark}

\newtheorem{subremark}[substuff]{Remark} 

\def\smap#1{\ {\buildrel #1 \over \rightarrow}\ }
\def\map#1{{~\buildrel #1 \over \longrightarrow~}}

\newcommand{\Spec}{\operatorname{Spec}}
\newcommand{\Pic}{\operatorname{Pic}}
\newcommand{\Br}{\operatorname{Br}}
\newcommand{\BW}{\operatorname{GB}}    
\newcommand{\RBr}{\operatorname{BR}}  
\newcommand{\RGB}{\operatorname{GBR}}
\newcommand{\hatZ}{\widehat{Z}^0}      
\newcommand{\id}{\operatorname{id}}

\newcommand{\End}{\operatorname{End}}
\newcommand{\Gm}{\mathbb{G}_{m}}

\begin{document}
\title[The Real Graded Brauer group]
{The Real Graded Brauer group}
\date{\today}

\author{Max Karoubi}
\address{Universit\'e Denis Diderot Paris 7 \\
Institut Math\'ematique de Jussieu --- Paris Rive Gauche}
\email{max.karoubi@gmail.com}
\urladdr{http://webusers.imj-prg.fr/~max.karoubi}

\author{Charles Weibel}
\address{Math.\ Dept., Rutgers University, New Brunswick, NJ 08901, USA}
\email{weibel@math.rutgers.edu}\urladdr{http://math.rutgers.edu/~weibel}
\thanks{Weibel was supported by NSA and NSF grants.}

\begin{abstract}
We introduce a version of the Brauer--Wall group for Real vector
bundles of algebras (in the sense of Atiyah), 
and compare it to the topological analogue of the Witt group.
For varieties over the reals, these invariants capture 
the topological parts of the Brauer--Wall and Witt groups.
\end{abstract}
\maketitle

\pagestyle{myheadings} 
\setcounter{section}{0}%

In this paper we introduce the Real graded Brauer group $\RGB(X)$
of a topological space with involution, and connect it to the
Real Witt group $WR(X)$ introduced in \cite{KSW}, 
via a Clifford algebra construction. These topological groups 
share many of the properties of their algebraic counterparts,
the Brauer--Wall group and the Witt group of an algebraic variety.
If $X$ is homotopic to a finite CW complex with involution, 
$\RGB(X)$ is a finite abelian group of exponent~8, and 
$WR(X)$ is a finitely generated abelian group.

When $V$ is an algebraic variety defined over $\R$, there are natural
maps from the Brauer--Wall group $\BW(V)$ and Witt group $W(V)$ to
$\RGB(V_\top)$ and $WR(V_\top)$, where $V_\top$ is the topological
space with involution associated to $V\oo_\R\C$. These maps are
compatible with the Clifford algebra maps in the algebraic and
topological settings. Moreover, $\RGB(V_\top)$ is the (finite) 
non-divisible part of $\BW(V)$, and 
the map $W(V)\to WR(V_\top)$ is close to an isomorphism
in low dimensions, depending on the Hodge number $h^{0,2}(V)$:

\begin{maintheorem}\label{thm:main}
1) 
If $V$ is a smooth curve, defined over either $\R$ or $\C$, 
then the map $\BW(V)\to\RGB(V_\top)$ is an isomorphism. 

\vspace{3pt}
\noindent 2) 
If $V$ is a smooth projective variety, defined over either $\R$ or $\C$, 
then there is a split exact sequence
\[
0\to (\bQ/\Z)^{\rho} \to \BW(V) \smap{\tau} \RGB(V_{\top}) \to 0,
\]
and $\tau: \BW(V) \cong \RGB(V_{\top})$ if and only 
$h^{0,2}(V)$ is~$0$.

\vspace{5pt}
\noindent 3)
If $V$ is a smooth projective surface defined over $\R$, 
with no $\R$-points, then: $WR(V_\top)$ embeds into $\RGB(V_\top)$;
there is a split exact sequence
\[
0\to (\Z/2)^{\rho} \to W(V) \smap{\tau} WR(V_{\top}) \to 0;
\]
and $\tau:W(V) \cong WR(V_{\top})$ if and only 
$p_g(V)=h^{0,2}(V)$ is~$0$.
\end{maintheorem}
\goodbreak\goodbreak

The number $\rho$ in Theorem \ref{thm:main} is the rank of the 
cokernel of the first Chern class $c_1:\Pic(V)\to \H_G^2(X,\Z(1))$,
where $\Z(1)$ is the sign representation of $G$; alternatively, 
$\rho$ is the rank of the map
$\H^2_G(X,\Z(1))\to\H^2_G(X,\cO_\an)$.

\noindent
Part 1 is proven in Corollary \ref{C-curves} and 
Proposition \ref{RGB:curves} below.
Part 2 is proven in Theorems \ref{h02=0} and \ref{real:pg=0}, and
Part 3 is proven in Theorems \ref{pg=0} and \ref{WR:pg=0}.
The obstruction in both the real and complex cases is the divisible subgroup 
$(\bQ/\Z)^{\rho}$ of the Brauer group of $V$.

\smallskip
The paper is organized as follows.
Section \ref{sec:BW} contains a review of the algebraic situation: 
the Brauer--Wall group $\BW(V)$ has a
filtration with associated graded groups $H_\et^0(V,\Z/2),H_\et^1(V,\Z/2)$ 
and the Brauer group $\Br(V)$. Note that by Cox' Theorem \cite{Cox},
$H_\et^n(V,\Z/2)$ is identified with the Borel cohomology 
$\H_G^n(V_\top,\Z/2)$.

In Sections \ref{sec:RGB}--\ref{sec:Q(X)}, 
we define the Real graded Brauer group $\RGB(X)$ and 
give it a filtration whose associated graded groups are:
$\H_G^0(X,\Z/2)$, $\H_G^1(X,\Z/2)$ and $\RBr(X)$, the Real analogue of the
Brauer group; in Section \ref{sec:RBr} we show that $\RBr(X)$ is
isomorphic to the torsion subgroup ${_\tors}\H_G^3(X,\Z(1))$
of $\H_G^3(X,\Z(1))$.

In Section \ref{sec:WR} we show that the
Clifford algebra map $WR(X)\to\RGB(X)$ is linked to the Stiefel--Whitney
classes of $X$, parallel to the algebraic setting.
Section \ref{sec:BW-RGB} establishes Part~2 of Theorem \ref{thm:main},
relating $\BW(V)$ and $\RGB(V_\top)$,
while Sections \ref{sec:C} and \ref{sec:free} establish Part~3, 
relating $W(V)$ and $WR(V_\top)$.

\smallskip

{\it Notation:}
$X$ will always denote a finite dimensional CW complex with involution,
with $\pi_0(X)$ finite. 
For example, if $V$ is quasi-projective variety over $\R$ then  $V_\top$
is finite dimensional but need not be compact.

A bundle whose fibers are vector spaces over $\R$ will be called a
{\it $\R$-linear vector bundle} to distinguish it from a Real bundle.

We shall write $G$ for the cyclic group of order~2 and
$\H_G^*(X,\Lambda)$ for the ($G$-equivariant) Borel cohomology 
of $X$ with coefficients in the $G$-module $\Lambda$.
By definition, this is the cohomology of the free $G$-space
$\XG=X\times_G EG$ with local coefficients $\Lambda$.

For example, $G$ acts by complex conjugation of the group
$\mu_n$ (resp., $\mu_\infty=\bigcup_n \mu_n$) of $n^{th}$ roots of unity
(resp., all roots of unity) in $\C$. We shall also write $\Z(1)$ 
and $\bQ(1)$ for the sign representations, 
so that $\bQ(1)/\Z(1)\cong\mu_\infty$.

\smallskip
\begin{center}{\it Acknowledgements}\end{center}
 
\smallskip
The authors are grateful to Sujatha, J.--L.\ Colliot-Th\'el\`ene 
and D. Krashen for several discussions.
We are are also grateful to 
J.\ Lannes for the argument of Theorem \ref{u=w_2(E)}, and to 
O.\ Wittenberg for help with
the computation of $\Br(V)$ in Theorem \ref{pg>0}.

\medskip\goodbreak
\section{The Brauer--Wall group of a scheme}\label{sec:BW}

The Brauer--Wall group $\BW(R)$ of a commutative ring $R$ containing $1/2$
was developed by Bass \cite[IV]{Bass-Tata} and Small \cite{Small} in
the late 1970s, following Wall \cite{Wall1963}, who introduced and
studied the Brauer--Wall group of a field.  
More recent references are the books by
Knus \cite[III.6]{Knus} and Lam \cite[IV]{Lam}. 

Although the definitions and results generalize easily to schemes 
containing $1/2$, we could not find any reference for this material.  
Therefore this section presents the Brauer--Wall group of a 
scheme $V$ whose structure sheaf $\cO=\cO_V$ contains $1/2$.

We say that a sheaf of $\cO_V$-algebras $\cA$ is {\it graded} if it is 
$\Z/2$-graded as a sheaf of rings, i.e., if $\cA=\cA_0\oplus\cA_1$, 
the product takes $\cA_i\oo\cA_j$ to $\cA_{i+j}$.
The graded tensor product $\cA\hat\oo\cB$ of two sheaves $\cA$, $\cB$ of 
graded algebras is the tensor product of the underlying graded
$\cO_V$-modules with multiplication determined by 
\begin{equation}\label{eq:hat-product} 
(a\oo b)(a'\oo b') = (-1)^{ij}(aa'\oo bb')  
\end{equation}
for homogeneous sections $a'$ of $\cA_i$, $b$ of $\cB_j$.
The opposite graded algebra $\cA^{\op}$ is $\cA$ with product 
$a\ast b=(-1)^{ij}ba$ for sections $a$ of $\cA_i$ and $b$ of $\cA_j$.

If $\cE=\cE_0\oplus \cE_1$ is a $\Z/2$-graded $\cO_V$-module,
its endomorphisms $\End(\cE)$ form a graded algebra;
the degree~0 subalgebra is $\End(\cE_0)\times\End(\cE_1)$.
We say that a locally free  sheaf $\cA$ of algebras is a 
(graded) {\it Azumaya algebra} if 
$\varphi:\cA\hat\oo\cA^{\op}\map{}\End(\cA)$ is an isomorphism,
where $\varphi(a\oo b)(c)=(-1)^{jk}acb$ for sections $a,b,c$ of
$\cA_i$, $\cA_j$, $\cA_k$.

\begin{defn}\label{def:BW}
The {\it Brauer--Wall group} of $V$, which we will write as $\BW(V)$, is the
set of equivalence classes of graded Azumaya algebras, where equivalence is
generated by making $\cA$ and $\cA\hat\oo\End(\cE)$ equivalent for every
$\Z/2$-graded locally free sheaf $\cE$. Equivalently,
graded Azumaya algebras $\cA$ and $\cB$ are equivalent 
if and only if $\cA\hat\oo\cB^\op$ is isomorphic to $\End(\cE)$
for some $\cE$.  It is an abelian group under
the product $\hat\oo$ of \eqref{eq:hat-product};
the identity element is $\cO_V$ concentrated in degree~0.
\end{defn}

\begin{example}\label{ex:Br}
A graded Azumaya algebra with $\cA_1=0$ is the same thing as
an ungraded Azumaya algebra. If $\cA=\cA_0$ and $\cA$ vanishes 
in $\BW(V)$ then $\cA_0\cong\End(\cE)$ for a 
graded locally free sheaf $\cE$ (necessarily concentrated in one degree),
and hence $\cA_0$ vanishes in $\Br(V)$.
Since $\cA\hat\oo\cB$ is just $\cA_0\oo\cB_0$ when $\cA_1=\cB_1=0$,
this induces a homomorphism from
the usual Brauer group $\Br(V)$ to $\BW(V)$.
Hence $\Br(V)$ injects into $\BW(V)$ as a subgroup.
When $V$ is quasi-projective, 
$\Br(V)$ is isomorphic to the torsion subgroup of $H_\et^2(V,\Gm)$; 
this is an unpublished result of Gabber; see \cite{deJong}.
\end{example}

\begin{example}\label{checkered}
Let us write $M_{1,1}$ for $\End(\cO_V\oplus\cO_V[1])$; it is the matrix
algebra $M_2(\cO_V)$ with the {\it checkerboard grading}: 
entries in the $(i,j)$-spot have degree $(-1)^{i+j}$. 
Any ungraded Azumaya algebra $\cA_0$ is equivalent to 
the graded Azumaya algebra $M_{1,1}\cA_0=M_{1,1}\hat\oo\cA_0$.
Thus (up to Brauer equivalence)
we can always tensor with $M_{1,1}$ to assume that $\cA_1\ne0$.
\end{example}

A {\it symmetric form} $(\cE,q)$ on $V$ is a locally free sheaf $\cE$
on $V$,  equipped with a nondegenerate
symmetric bilinear form $q:\cE\oo \cE\to \cO_X$. 
They form a monoidal category, where the monoidal operation is
$(\cE_1,q_1)\perp(\cE_2,q_2) = (\cE_1\oplus \cE_2, q_1\oplus q_2).$
The Grothendieck--Witt group $\GW(V)$ is the Grothendieck group of
the monoidal category of symmetric forms $(\cE,q)$,
modulo the relation that $[(\cE,q)]=[h(\cL)]$ 
if $\cE$ has a Lagrangian $\cL$ (a subobject such that $\cL=\cL^\perp$);
$h(\cL)$ is the hyperbolic form on $\cL\oplus\cL^*$.

\begin{defn}\label{def:Witt}  
The Witt group $W(V)$ is the cokernel of the hyperbolic map
$K_0(V)\to \GW(V)$ sending $[\cE]$ to the class of its
associated hyperbolic form $h(\cE)$ on $\cE\oplus\cE^*$.

The {\it Clifford algebra} $C(q)=\Cl(\cE,q)$ of a symmetric form
$(\cE,q)$ is a graded algebra, equipped with a morphism $\cE \to C_1(q)$
so that $q$ composed with $\cO_X\!\subseteq\!C_0(q)$
is $\cE\oo \cE \smap{} C_1(q)\oo C_1(q) \smap{} C_0(q)$.
In fact, $C(q)$ is a graded Azumaya algebra on $X$.
By Knus--Ojanguren \cite{KO}, $C(q)$ vanishes in $\BW(V)$ for 
quadratic forms with a Lagrangian, so there is a
Clifford algebra functor $GW(V)\to\BW(V)$. It induces a homomorphism
$\Cl:W(V)\to\BW(V)$, because every hyperbolic form $h(\cE)$
has a Lagrangian; in fact, $\Cl(h(\cE))\cong\End(\wedge^*\cE)$.
\end{defn}

\goodbreak
\medskip{\bf Quadratic algebras}\medskip

\noindent
A {\it quadratic algebra} is a commutative $\cO_V$-algebra
of the form $Q=\cO_V\!\oplus\! L$, where $L\cong Q/\cO_V$ is an 
invertible sheaf, equipped with an isomorphism $q: L\oo L\cong\cO_V$.
(Multiplication in $Q$ is defined by $q$.)
The set of isomorphism classes of quadratic algebras is denoted $\Q(V)$.

A quadratic algebra defines an \'etale double cover of $V$.
The prototype is when $L$ is $\cO_V$, generated by a global section $z$,
and $z^2=a$ is a unit in $\cO(V)$;  we write
$\cO[z]/(z^2=a)$ for this algebra. Since $a$ is independent 
of the choice of $z$ up to squares, the map $a\mapsto \cO[z]/(z^2=a)$
defines an injection of $\cO(V)^\times/\cO(V)^{\times2}$ into $\Q(V)$.

In fact, $\Q(V)$ is isomorphic to $H_\et^1(V,\Z/2)$, with the product
$Q\circ Q' = \cO_V\times(L\oo L')$ in $\Q(V)$ corresponding to addition;
the identity element of $\Q(V)$ is $\cO[z]/(z^2-1)$. 
The surjection $\Q(V)\cong H_\et^1(V,\Z/2)\to{_2}\Pic(V)$
is the map $Q\mapsto L$.

A {\it graded quadratic algebra} is a graded $\cO$-algebra
of the form $Q=\cO\oplus L$, where $L$ is an invertible sheaf
(either in degree~0 or~1)
equipped with a product $L\oo L\map{\cong}\cO$.
A graded quadratic algebra concentrated in degree~0 (i.e., with $Q_1=0$)
is just an ordinary quadratic algebra; if $Q_1=L$, the product is
skew-symmetric.  

\begin{defn}\label{def:*}
We write $\Q_2(V)$ for the set of
isomorphism classes of graded quadratic algebras over $V$. 
There is a commutative product $\ast$ on $\Q_2(V)$ defined by 
letting $(\cO_V\oplus L)\ast (\cO_V\oplus L')$ be 
$\cO_V \oplus (L\oo L')$, with the isomorphism
$(L\oo L')\oo (L\oo L')\to\cO$ given by the usual sign convention.
This makes $\Q_2(V)$ into an abelian group.
\end{defn}

\begin{example}\label{ex:*} 
If $z_1$ and $z_2$ are in degree~1, 
and $z$ is in degree~0, then
\[
\cO[z_1]/(z_1^2=a_1)\ast\cO[z_2]/(z_2^2=a_2) =
\cO[z]/(z^2=-a_1a_2).
\]
\end{example}

By construction, there is a homomorphism $\pi:\Q_2(V)\to H^0(V,\Z/2)$
sending $Q_0\oplus Q_1$ to the rank of $Q_1$ (0 or 1 on each 
connected component of $V$), and the kernel of $\pi$
is isomorphic to $\Q(V)$. We may consider $\Q_2(V)\cong H_\et^1(V,\Z/2)$ 
to be a ring in which $\Q(V)$ is a square-zero ideal. In summary,
there is an extension
\begin{equation}\label{eq:Q_2}
0 \to H_\et^1(V,\Z/2) \to \Q_2(V) \map{\pi} H_\et^0(V,\Z/2) \to 0.
\end{equation}
(See \cite[IV.3.3]{Bass-Tata} or \cite[6.2.2]{Knus}.)
\medskip

The {\it graded centralizer} of a subsheaf $S$ of 
homogeneous elements in $\cA$ is the $\cO_V$-algebra 
$\widehat{C}(S)=\widehat{C}_0\oplus\widehat{C}_1$
such that $\widehat{C}_i$ commutes with $S_j$ up to $(-1)^{i+j}$.
We write  $\hatZ(\cA)$ for the graded centralizer of  $S=M_{1,1}\cA_0$
in the algebra $M_{1,1}\cA$ of Example \ref{checkered}.
If $\cA_1$ has strictly positive rank, then the graded centralizer 
$\widehat{C}(\cA_0)$ of $\cA_0$ in $\cA$ is isomorphic to $\hatZ(\cA)$.

\begin{theorem}\label{BW}
If $\cA$ is a graded Azumaya algebra, then
$\hatZ(\cA)$ is a graded quadratic algebra, and
$\hatZ:\BW(V)\to\Q_2(V)$ is a homomorphism:
$\hatZ(\cA\hat\oo \cB)\cong\hatZ(\cA)\ast\hatZ(\cB)$.
The inclusion of Example \ref{ex:Br}
and the function $\hatZ$ fit into an exact sequence
\[
0 \to \Br(V) \to \BW(V) \map{\hatZ} \Q_2(V) \to 0.
\]
Thus the group $\BW(V)$ has a filtration by subgroups whose associated
graded groups are
\[ 
H^0(V,\Z/2),\quad H_\et^1(V,\Z/2) \quad\textrm{~and~}\quad 
\Br(V)\cong{_\tors}H_\et^2(V,\Gm).
\]
\end{theorem}

The extension in 
Theorem \ref{BW} can be non-trivial;
for example, if $V=\Spec(\R)$ we have $\BW(V)=\Z/8$.

Theorem \ref{BW} was first asserted by Bass in \cite[IV.4.4]{Bass-Tata}
when $V=\Spec(R)$, using $\widehat{C}(\cA_0)$; this was corrected
in \cite[7.10]{Small}.
Another reference for this result (in the affine setting) is
Knus' book, \cite[6.4.7]{Knus}. We remark that Wall always assumes
that $\cA_1\ne0$ in \cite{Wall1963}.

\begin{remark}\label{BW-Br}
Suppose for simplicity that $V$ is connected.
The composition $\pi\circ\hatZ:\BW(V)\to\Z/2$ is the parity;
we say that $\cA$ is {\it even} or {\it odd}, according to
whether its parity is~0 or~1. 

There is a set-theoretic section $u:\BW(V)\to\Br(V)$
of the inclusion. If $\cA$ is even, forgetting the
grading yields an Azumaya algebra which we write as $u\cA$;
if $\cA$ is odd, $\cA_0$ is an Azumaya algebra and we set
$u\cA=\cA_0$. This was first observed by Wall in \cite[Thm.\,1]{Wall1963}.
\end{remark}


\bigskip
\section{The Real Brauer group}\label{sec:RGB}

In this section, we introduce Real vector bundles, 
the Real Brauer group $\RBr(X)$ and the Real graded Brauer group $\RGB(X)$.

Let $X$ be a topological space with involution $\sigma$.
By a {\it Real vector bundle} on $X$ we mean a complex vector bundle
$E$ with an involution $\sigma$ compatible with the involution on $X$
and such that for each $x\in X$ the isomorphism
$\varphi:E_x\to E_{\sigma x}$ is $\C$-antilinear. 
Following Atiyah \cite{Atiyah}, we write $KR(X)$
for the Grothendieck group of Real vector bundles on a compact space $X$.
Since the tensor product of Real vector bundles is a Real vector bundle,
$KR(X)$ is a ring.  

We write $\C_X$ for the trivial Real vector
bundle $X\times\C$ with $\sigma(x,z)=(\sigma(x),\bar{z})$; 
$E\oo\C_X\cong E$ for all $E$. We will write $\C_X(1)$ for 
the complex vector bundle $X\times\C$ with the Real vector bundle 
structure $\sigma(x,z)=(\sigma(x),-\bar{z})$; 
$\C_X$ and $\C_X(1)$ differ if $G$ has a fixed point $x$ on $X$, 
as they have different signatures there.

By a {\it Real algebra} on $X$ we mean an algebra object in the
category of Real vector bundles, i.e., a Real vector bundle $A$
with a map $A\oo A\to A$ and a global section $1$ making each fiber
$A_x$ into a $\C$-algebra, with each $A_x\to A_{\sigma x}$
a map of $\R$-algebras. The tensor product of Real algebras on $X$ is the
tensor product of the underlying Real vector bundles.

If $E$ is a Real vector bundle, then
$\End(E)$ is a Real algebra whose fibers are matrix algebras;
the map $\End(E_x)\to \End(E_{\sigma x})$ sends $\eta$ to
$a\mapsto\overline{\eta(\bara)}$; see \cite{Atiyah}.

\begin{defn}
A {\it Real Azumaya algebra} on $X$ is a Real algebra $A$ such that
$A\oo A^{\op}\map{\cong}\End(A)$.
The {\it Real Brauer group} $\RBr(X)$ is the set of equivalence
classes of Real Azumaya algebras on $X$,  where equivalence is
generated by making $A$ and $A\oo\End(E)$ equivalent for every
Real vector bundle $E$. It is an abelian group under
the product $\oo$ with identity $\C_X$.
\end{defn}

A {\it Real graded algebra} $A=A_0\oplus A_1$ on $X$ is just a 
$\Z/2$-graded algebra object in the category of Real vector bundles, 
and the graded tensor product $A\hat\oo B$ is defined as in 
\eqref{eq:hat-product}.
A {\it Real graded Azumaya algebra} on $X$ is a Real graded algebra 
$A=A_0\oplus A_1$ such that $\varphi:A\hat\oo A^{\op}\smap{\cong}\End(A)$.
(As in Section \ref{sec:BW}, $A^\op$ is the Real graded bundle
with product $a\ast b=(-1)^{ij}ba$, and $\varphi(a\oo b)(c)=(-1)^{jk}acb$.)

\begin{defn}\label{def:RGB}
The {\it Real graded Brauer group} $\RGB(X)$ is the set of equivalence
classes of Real Azumaya algebras on $X$, where `equivalence' is
generated by making $A$ and $A\hat\oo\End(E)$ equivalent for every graded
Real vector bundle $E$. It is an abelian group with the
product $\hat\oo$; as in the definition
\ref{def:BW} of $\BW$, the identity is $\C_X$. As in Example \ref{checkered},
$\C_X$ is equivalent to $M_{1,1}\C_X$ with the checkerboard grading.
\end{defn}

\begin{example}\label{ex:RBr}
A graded Azumaya algebra with $A_1=0$ is the same as 
an ungraded Azumaya algebra on $X$.  As in Example \ref{ex:Br},
these form a subgroup of $\RGB(X)$ isomorphic to the
Real Brauer group $\RBr(X)$.
\end{example}

\begin{example}\label{Br-RBr}
If $V$ is a variety over $\R$, the complex analytic space $V_\top$
associated to $V(\C)$ has an involution (complex conjugation)
and there is a natural homomorphism $\Br(V)\to\RBr(V_\top)$.
Indeed, every locally free sheaf $\cE$ on $V$ naturally determines a 
Real vector bundle $E$ on $V_\top$, and 
$\End(\cE)$ determines $\End(E)$. This is compatible with products,
so every Azumaya algebra $\cA$ on $V$ naturally determines a 
Real Azumaya algebra $A$ on $V_\top$.

The same construction shows that 
every graded Azumaya algebra $\cA$ on $V$ naturally determines a 
Real graded Azumaya algebra $A$ on $V_\top$,
compatible with products. Hence we have a natural homomorphism 
$\BW(V)\to\RGB(V_\top)$, compatible with the map $\Br(V)\to\RBr(V_\top)$.
\end{example}

\medskip\begin{center}{\it Connection to $WR$}\end{center}\medskip

\noindent
Recall from \cite{KSW} that the Real Grothendieck-Witt group $GR(X)$ is
the Grothendieck group of the monoidal category of 
symmetric forms $(E,q)$, where $E$ is a Real vector bundle
on $X$ and $q:E\oo E\to\C_X$ is a nondegenerate symmetic bilinear form,
and the monoidal operation is
$(E_1,q_1)\perp(E_2,q_2) = (E_1\oplus E_2, q_1\oplus q_2).$

\begin{example}\label{ex:Clifford}
Let $(E,q)$ be a Real vector bundle equipped with a symmetric bilinear form
$q:E\oo E\to \C_X$.
The (complex) {\it Clifford algebra} $C(q)=\Cl(E,q)$ is
a Real graded algebra, equipped with a morphism $E \to C_1(q)$
so that $q$ composed with $\C_X\!\subseteq\!C_0(q)$
is the form $E\oo E \smap{} C_1(q)\oo C_1(q) \smap{} C_0(q)$.
Since the fiber of $C(q)$ over $x$ is the usual Clifford algebra 
of $(E_x,q_x)$, each fiber is a graded Azumaya algebra over $\C$.
It follows that $C(q)$ is a Real graded Azumaya algebra on $X$.
In fact, $\Cl$ is a homomorphism $GR(X)\to\RGB(X)$.
\end{example}

\begin{defn}
The Real Witt group $WR(X)$ of a compact $G$-space $X$
is the cokernel of the hyperbolic map $KR(X)\smap{h} GR(X)$,
sending $[E]$ to the class of its associated hyperbolic form, $h(E)$.

The map $GR(X)\to \H_G^0(X,\Z)$, sending $(E,q)$ to its rank,
induces a map $WR(X)\to \H_G^0(X,\Z/2)\cong H^0(X/G,\Z/2)$; we write 
$I(X)$ for the kernel, which consists of forms of even rank.
\end{defn}

\begin{subremark}\label{GR=KOG}
$WR(X)$ is also the cokernel of the forgetful map $u:KR(X)\to KO_G(X)$.
Indeed, by \cite[Thm.\,2.2]{KSW}, there is an isomorphism $GR(X)\cong KO_G(X)$
identifying the hyperbolic map with $u$.
\end{subremark}


\begin{lemma}\label{C-homom}
The Clifford algebra map $\Cl:GR(X)\to\RGB(X)$ induces 
a homomorphism $WR(X) \to \RGB(X)$.
\end{lemma}

\begin{proof}
The tensor product $\Cl(q_1\perp q_2)$ is isomorphic to
$\Cl(q_1)\hat\oo \Cl(q_2)$, because it is true fiberwise.
This proves that the functor $\Cl$ induces a group homomorphism 
from $GR(X)$ to $\RGB(X)$. The composition $KR(X)\to\RGB(X)$
is trivial because it sends the class of a Real vector bundle
$E$ to $\End(\wedge^*E)$, 
where $\wedge^*E$ is the exterior algebra of $E$.
\end{proof}

\medskip\begin{center}{\it Two test cases}\end{center}\medskip

\begin{example}\label{fixed:RGB}
When the involution on $X$ is trivial, every Real vector bundle
has a canonical form $E\oo\C$, where $E$ is an $\R$-linear bundle
and the involution acts by complex conjugation.  In this case, an
Azumaya algebra has the form $A'\oo\C$, where $A'$ is an
$\R$-linear algebra bundle on $X$, and similarly for Real graded
Azumaya algebras.  The groups of equivalence classes of the $A'$
are isomorphic to $\RBr(X)$ and $\RGB(X)$; they were studied by 
Donovan and Karoubi, who showed in \cite[Thms.\,3,6]{DonovanKaroubi}
that (for connected $X$)
$$\RBr(X)\cong\Z/2\oplus H^2(X,\Z/2),$$ 
$\RGB(\pt)=\Z/8$ is a summand of $\RGB(X)$, and there is an extension
\begin{equation*}
0 \to H^2(X,\Z/2) \to \RGB(X)\to \Z/8 \oplus H^1(X,\Z/2) \to 0.
\end{equation*}

\noindent
By \cite[pp.\,10-11]{DonovanKaroubi}, if $a\in H^1(X,\Z/2)$ then
$a+a=\beta(a)\in H^2(X,\Z/2)$, where 
$\beta:H^1(X,\Z/2)\to H^2(X,\Z/2)$ is the Bockstein. Thus the number of 
$\Z/4$ summands in $\RGB(X)$ equals the rank of the Bockstein.  
For example, this implies that $\RGB(\R\bP^2)\cong\Z/8\oplus\Z/4$.
By way of comparison, we showed in \cite[Ex.\,2.5]{KSW} that 
$WR(\R\bP^2)\cong\Z\oplus\Z/4$. 

If $S^n$ denotes the $n$-sphere for $n>2$, with $G$ acting trivially, 
then $\RGB(S^n)\cong\Z/8$.
If $H$ is the rank~2 $\R$-linear $G$-bundle on $S^2$ underlying the
canonical (Hopf) complex line bundle, then its Clifford algebra
$\Cl(H)$ is the nontrivial element of $\RBr(S^2)$
vanishing on $\RBr(\pt)\cong\Z/2$.
\end{example}

\begin{proposition}\label{free:RGB}
If $X=G\times Y$ then 
\[
\RGB(X)\cong H^0(Y,\Z/2)\oplus H^1(Y,\Z/2)\oplus\ {_\tors}H^3(Y,\Z).
\]
\end{proposition}

\begin{proof}
Real vector bundles on $X=G\times Y$ are the same as 
complex vector bundles on $Y$.
Thus $\RGB(X)$ recovers the graded Brauer group
$\textrm{GBrU}(Y)$ studied by Donovan and Karoubi in \cite{DonovanKaroubi}.
For any finite CW complex $Y$, Theorem 11 of {\it loc.\,cit.} states
that $\textrm{GBrU}(Y)$ is the direct sum of $H^0(Y,\Z/2)$ and an
extension $HU(X)$ of $H^1(Y,\Z/2)$ by the torsion subgroup
of $H^3(Y,\Z)$, equipped with a canonical section 
$i:H^1(Y,\Z/2)\to\textrm{GBrU}(Y)$.

To see that the extension splits, 
recall that, if $\tilde\beta$ is the integral Bockstein and 
$a,b\in H^1(Y,\Z/2)$, $i(a)+i(b)=i(a+b)+\tilde\beta(a\cup b)$
in $H^3(Y,\Z)$ by \cite[pp.\,10-11]{DonovanKaroubi}.
Since the cohomology operation $\tilde\beta\;\Sq^1$ is zero,
$i(a)+i(a)=\tilde\beta\;\Sq^1(a)=0$, and the extension splits.
A similar elementary argument shows that 
$\tilde\beta(a\cup b)=0$, so $i$ is a homomorphism.
\end{proof}

\begin{corollary}\label{C-curves}
If $V$ is a connected algebraic curve defined over $\C$, then
\[
\BW(V) \cong \RGB(V_\top) \cong \Z/2\oplus H^1_\et(V,\Z/2).
\]
\end{corollary}

Indeed, we have $V_\top\cong Y\times G$, where $Y=V(\C)$ is 2-dimensional.
In this case, $\Br(V)=0$ and the result follows from Cox' Theorem, 
Theorem \ref{BW} and Proposition \ref{free:RGB}.

\section{Real quadratic algebras}\label{sec:Q(X)}

In this section, $X$ will be a $G$-space.
%
By analogy with the algebraic setting, a commutative Real algebra $Q$ 
on $X$ of rank~2 is called a {\it Real quadratic algebra} if it has 
the form $\C_X\oplus L$ for a rank 1 Real vector bundle $L$ equipped with
an equivariant isomorphism $\theta:L\oo L\map{\cong}\C_X$.

\begin{defn}
The set $\Q(X)$ of isomorphism classes of Real quadratic algebras
becomes a group with product $\circ$, as in Section \ref{sec:BW}.
The identity of $\Q(X)$ for $\circ$ is the Real quadratic algebra
$Q^{(+)}\!=\C_X[u]/(u^2=1)$ with underlying Real bundle $\C_X\oplus\C_X$.
\end{defn}

\begin{subex}\label{Q+/-}
There is another Real quadratic algebra, $Q^{(-)}$, which is the 
algebra $\C_X[u]/(u^2=1)$ with underlying Real bundle
$\C_X\oplus\C_X(1)$; $Q^{(-)}$ has order~2 in $\Q(X)$.
It is useful to note that if $t=iu$ then $Q^{(-)}$ is the algebra
$\C_X[t]/(t^2=-1)$ with underlying Real bundle $\C_X\oplus\C_X$.
\end{subex}

We write $\Pic_G(X)$ for the group of rank~1 $\R$-linear
$G$-vector bundles on $X$ with product $\otimes$;
its unit `$1$' is the trivial line bundle $\R_X$. 
If $L_0$ is such a bundle,  $L_0\otimes L_0$ is trivial.
The proof of the following result
is inspired by \cite[Prop.\,1]{Kahn}.


\begin{proposition}\label{w1-bundles} 
For every $X$, 
$\Q(X)\cong\Pic_G(X)\smap{\cong}\H^1_G(X,\Z/2)$.
\end{proposition}

\begin{proof} 
We proved in \cite[2.2]{KSW} that a quadratic form $\theta$ on a 
Real line bundle $L$ determines%
\footnote{Fix a $G$-invariant Hermitian metric on $L$ and define
$T$ by $\langle Tu,v\rangle=\theta(u,v)$; $L_0$ is the family 
of +1 eigenspaces of $T$. 
The bilinear map $\theta$ is determined up to isomorphism by the choice
of an equivariant Riemannian metric on $L_0$.}
an equivariant $\R$-linear line bundle $L_0$
such that $L\cong L_0\oo\C$; the map $\sigma: L_x\to L_{\sigma x}$
sends $\lambda\oo z$ to $\sigma(\lambda)\oo\bar{z}$.
If $L=L_0\oo\C$, we recover $L_0$.
Thus $\Pic_G(X)$ and $\Q(X)$ are isomorphic.

To classify $\Pic_G(X)$, let
$\cO_\R(U)$ denote the group of continuous $\R$-valued functions on $U$;
$\cO_\R$ is a $G$-sheaf.  The usual description of a bundle using
(equivariant) \v{C}ech cocycles in $\check{\H}^1_G(X,\cO_\R^\times)$ 
as patching data (and Cartan's criterion \cite[III.2.17]{Milne})
shows that there is an
isomorphism $\Pic_G(X)\to \H^1_G(X,\cO_\R^\times)$ of abelian groups.
Consider the exact exponential sequence of equivariant sheaves:
\[ 0\to \cO_\R ~\map{\exp}~ \cO_\R^\times \to \Z/2 \to 0.  \]
Since the first sheaf is soft, hence acyclic, we get an isomorphism from
$\Pic_G(X)\cong H^1(X\!\times_G EG,\cO_\R^\times)$ to 
$H^1(X\!\times_G EG,\Z/2)\cong \H^1_G(X,\Z/2)$.
\end{proof}


\begin{example}\label{ex:sign}
Let $x\in X$ be a fixed point, and $L_0$ an $\R$-linear line bundle on $X$.
The involution on the fiber of $L_0$ over $x$ must be
multiplication by $\pm1$. 
In particular, $\Pic_G(\pt)\cong\{\pm1\}$.
If $X^G$ has $\nu$ components, this induces a natural {\it sign map} 
\[
\Pic_G(X)\map{\mathrm{sign}}\{\pm1\}^\nu \cong H^0(X^G\!,\Z/2).
\]
It is the composition of $\H^1_G(X,\Z/2)\to\H^1_G(X^G,\Z/2)$ with
\[ 
\H^1_G(X^G\!,\Z/2)\cong H^1(X^G\!\times\! BG,\Z/2)\cong
\H^0(X^G\!,\Z/2)\oplus H^1(X^G\!,\Z/2).
\]
\end{example}

\begin{subremark}\label{sign-onto}
The sign map is not always onto when $\nu\ge2$;
it is onto when $\nu=1$, as $X\to\pt$ induces the splitting
$\Pic(\pt)\to\Pic_G(X)$.  The sign map is onto when $X=X^G$,
by Example \ref{fixed:RGB}. 
It is also onto if $\dim(X)=1$,
since in that case $\H_G^1(X,\Z/2)\to\H_G^1(X^G,\Z/2)$ is onto.
\end{subremark}

\begin{lemma}\label{H_G^1}
Let $\Pic^0_G(X)$ denote the kernel of the sign map.
There is a natural isomorphism $w:\Pic^0_G(X)\smap{\cong} H^1(X/G,\Z/2)$
and hence an exact sequence
\[
0 \to H^1(X/G,\Z/2) \to \H^1_G(X,\Z/2) \map{\mathrm{sign}} \{\pm1\}^\nu.
\]
\end{lemma}

\begin{proof}
The map $w$ is defined as follows.
If $L$ is a rank~1 $\R$-linear $G$-bundle on $X$ with trivial sign, the 
identifications $L_x\cong L_{\sigma x}$ imply that $L$
descends to a line bundle $L/G$ on $X/G$, and the rank~1 $\R$-linear
bundle $L/G$ on $X/G$ is classified by an element $w(L)$ of
the group $H^1(X/G,\Z/2)$.
Conversely, $w\in H^1(X/G,\Z/2)$ determines a line bundle on $X/G$,
and its pullback along $X\to X/G$ is an equivariant line bundle on $X$ 
with trivial sign, i.e., an element of $\Pic_G^0(X)$.
\end{proof}

\begin{subex}\label{Xcovers}
Let $\R_X(1)$ denote the trivial $\R$-linear line bundle 
with $G$ acting by $-1$. 
By Proposition \ref{w1-bundles}, the Real quadratic algebra 
corresponding to $\R_X(1)$ is the algebra $Q^{(-)}$
of Example \ref{Q+/-}.

If $X$ is connected, then $\R_X(1)$ is nontrivial in $\Pic_G(X)$. 
This is clear if $X^G\ne\emptyset$, as the sign of $\R_X(1)$ is
$-1$ on each component of $X^G$.
When $X^G\!=\emptyset$, it is a nonzero element of 
$\Pic^0_G(X)\cong H^1(X/G,\Z/2)$. This is because $X\to X/G$ is 
a nontrivial covering space (as $X$ is connected), covering spaces
with group $G$ are classified by elements of $H^1(X/G,\Z/2)$, and
we showed in Example 2.7 of \cite{KSW} that the isomorphism of
Proposition \ref{w1-bundles} sends $\R_X(1)$ 
to the element classifying this particular cover.

The hypothesis that $X$ be connected is necesssary;
when $X=Y\times G$, $\R_X(1)$ is trivial in
$\Pic_G(X)\cong H^1(Y,\Z/2)$. 
\end{subex}

\begin{subex}\label{theta=t}
Let $T=S^{1,1}$ be the unit circle in $\C$, with the induced complex
conjugation as involution. 
The trivial Real line bundle $\C_T$ carries a canonical
Real symmetric form $\theta=t$, which is multiplication by $t$ on 
the fiber over $t\in T$.  The corresponding quadratic algebra is
$\C_X[z]/(z^2=t)$, and the corresponding nontrivial element of
$\Pic_G(T)\cong\{\pm1\}^2$ is the $\R$-linear subbundle of $\C_T$ 
whose fiber over $t$ is $it\cdot\R$.
The sign of this element is different at the two fixed points of $T$.
\end{subex}

\smallskip
\begin{defn}\label{defn:Q(X)}
A {\it Real graded quadratic algebra} on $X$ is a graded-commutative
Real algebra $Q=Q_0\oplus Q_1$ of the form $\C_X\oplus L$ for a 
Real vector bundle $L$ of rank~1, equipped with
an isomorphism $L\oo L\map{\cong}\C_X$.
($L$ may be in degree~0 or~1.)

We shall write $\Q_2(X)$ for the set of isomorphism classes
of Real graded quadratic algebras on $X$. 
As in Definition \ref{def:*}, there is a commutative product $\ast$
on $\Q_2(X)$ making it into a commutative group, with
identity $Q^{(+)}$.
By definition, $\Q(X)$ is a subgroup of $\Q_2(X)$. 
\end{defn}

There is a homomorphism 
$\Q_2(X) \smap{\pi} \H_G^0(X,\Z/2)\cong H^0(X/G,\Z/2)$ 
sending a Real graded quadratic algebra
to the rank of its degree~1 component.

\begin{example}\label{ex:Z/4}
We have $\pi(\Clone)=1$, where 
$\Clone$ denotes the Clifford algebra of the quadratic bundle 
$(\C_X,1)$ with its usual grading. 
As in Example \ref{ex:*}, $\Clone\ast\Clone$ is the Real
quadratic algebra $Q^{(-)}$ of Example \ref{Q+/-}
concentrated in degree~0.

For completeness, we note that $\Clone\ast Q^{(-)}$ is the
Clifford algebra $\Cl_{\langle-1\rangle}$ of the quadratic bundle 
$(\C_X,-1)\cong(\C_X(1),1)$, and $\Clone\ast\Cl_{\langle-1\rangle} = Q^{(+)}$.
\end{example}

Here is the analogue of the extension \eqref{eq:Q_2}.
The proof shows that the group operation $\ast$ in $\Q_2(X)$ is 
determined by the fact that $\Clone\ast\Clone$ is the 
Real quadratic algebra $Q^{(-)}$ of Example \ref{Q+/-}. 

\begin{proposition}\label{ext:Q_2}
There is a group extension
\begin{equation*}
0 \to \H_G^1(X,\Z/2) \to \Q_2(X) \map{\pi} \H_G^0(X,\Z/2) \to 0.
\end{equation*}
If $X$ is connected, so $\H_G^0(X,\Z/2)\cong\Z/2$, 
this is a nontrivial extension, i.e., $\Q_2(X)$ has exponent~4.
In particular, $\Q_2(\pt)\cong\Z/4$.
\end{proposition}

\begin{proof}
The sequence is exact at the first two spots by
Proposition \ref{w1-bundles}, because
$\pi(Q)=0$ if and only if $Q$ is a Real quadratic algebra.
By Example \ref{ex:Z/4}, $\pi(\Clone)=1$ (so $\pi$ is onto)
and 
$\Clone\ast\Clone$ is the Real quadratic algebra $Q^{(-)}$.
By Example \ref{Xcovers}, $Q^{(-)}$ corresponds 
to $\R_X(1)$, and is nontrivial when $X$ is connected.
\end{proof}


Let $M_{1,1}A$ be the checkerboard-graded algebra of 
Example \ref{checkered}, associated to a Real graded algebra 
$A$ on $X$.
The graded centralizer $\hatZ(A)$ of $S=M_{1,1}\,A_0$ in $M_{1,1}A$ is a 
Real graded quadratic algebra, by Theorem \ref{BW} applied to
each fiber.  As in {\it loc.\,cit.}, if $A_1$ is nowhere zero then
the graded centralizer of $A_0$ in $A$  is isomorphic to $\hatZ(A)$.

\goodbreak
\begin{theorem}\label{RBW}
$\hatZ(A)$ is a Real graded quadratic algebra, and
$\hatZ$ is a homomorphism:
$\hatZ(A\hat\oo B)\cong\hatZ(A)\ast\hatZ(B)$.
The inclusion $i$ of Example \ref{ex:RBr}
and the function $\hatZ$ fit into an exact sequence:
\[
0 \to \RBr(X) \map{i} \RGB(X) \map{\hatZ} \Q_2(X) \to 0.
\]
\end{theorem}

\begin{proof}
The first sentence follows from Theorem \ref{BW} because
each fiber of $\hatZ(A)$ is a graded quadratic algebra,
compatible with the graded tensor product $\hat\oo$,
and the involution preserves centralizers.

This shows that the displayed sequence is exact at $\RGB(X)$,
suppose that $A$ is a Real graded algebra on $X$ with
$\hatZ(A)=Q^{(+)}$ 
in degree~0. By Theorem \ref{BW},
each fiber $A_x$ is a graded Azumaya algebra over $\C$ with
$\hatZ(A_x)=\C$. That is, $A$ is a Real Azumaya algebra.

To show that $\hatZ$ is onto, let $Q=Q_0\oplus Q_1$
be a Real graded quadratic algebra on $X$.  If $Q_1\ne0$,
then $Q$ is also a Real graded Azumaya algebra, and
$\hatZ(Q)=Q$.  If $Q_1=0$ then $Q=Q_0=\C_X\oplus L$ and
we consider $A=Q\oplus Qu$ with $u^2=-1$ and 
$\lambda u=-u\lambda$ for $\lambda\in L$. This is a Real
graded Azumaya algebra on $X$ with $\hatZ(A)=Q$;
see \cite[I(1.3.7)]{Knus}. 
\end{proof}

\begin{subremark}\label{RGB-RBr}
Suppose for simplicity that $X$ is connected.
The composition $\pi\circ\hatZ:\RGB(X)\to\Z/2$ is the {\it parity};
we say that $A$ is {\it even} or {\it odd}, according to
whether its parity is~0 or~1. 

There is a set-theoretic section $u:\RGB(X)\to\RBr(X)$
of the inclusion in Example \ref{ex:RBr}. If $A$ is even, forgetting the
grading yields an Azumaya algebra which we write as $uA$;
if $A$ is odd, $A_0$ is an Azumaya algebra and we set $u\cA=\cA_0$.
By Remark \ref{BW-Br},  $u$ is well defined.

If $X$ is a point, $u$ is the section $\Z/8\to\Z/2$ of 
$\RBr(\pt)\subset\RGB(\pt)$.
\end{subremark}

\begin{lemma}\label{WR/I2}
The composition $WR(X) \smap{\Cl} \RGB(X) \smap{\hatZ} \Q_2(X)$ is onto,
and $WR(X)\to\Q_2(X)\to\H_G^0(X,\Z/2)$ is the rank mod~2.
\end{lemma}

\begin{proof}
Since we saw in Example \ref{ex:Z/4} that the Clifford algebra $\Clone$
of $(\C_X,1)$ has $\pi(\Clone)=1$, it suffices to show that
$I(X)$ maps onto $\Q(X)$.

Suppose that $C=C_0\oplus C_1$ is the Clifford algebra of a 
Real quadratic space of even rank on $X$. To show that 
$\hatZ(C)$ is in $\Q(X)$, we may proceed fiberwise.
By Knus \cite[IV(2.2.3)]{Knus}, each fiber $C_x$ is a 
graded Azumaya algebra of even type.
In particular, $\hatZ(C_x)$ is a quadratic algebra over $\C$;
it follows that $\hatZ(C)$ is a Real quadratic algebra on $X$.

To show that the map is onto, let $\C_X\oplus L$ be a 
Real quadratic algebra with a symmetric form $\theta$ on $L$,
and consider the Real vector bundle $E=\C_X\oplus L$ with 
symmetric form $q=(1\oplus-\theta)$ 
The Clifford algebra $\Cl(E,q)$ has
$\hatZ(\Cl(E,q))=\hatZ(\Cl(\C_X,1))\ast\hatZ(\Cl(L,-\theta))$.

The Clifford algebra $Q=\Cl(L,-\theta)$ also has $\pi(Q)=1$,
and $Q\ast\Clone$ is the Real quadratic algebra $\C_X\oplus L$,
with $L\oo L\to\C_X$ given by $\theta$ (see Definition \ref{def:*}). 
Now $\hatZ\circ\Cl$ is a group homomorphism,
by Lemma \ref{C-homom} and Theorem \ref{RBW}, so
$I(X)$ maps onto $\Q(X)$.
\end{proof}

\begin{lemma}\label{Cox-H1}
If $V$ is a variety over $\R$, then $\Q_2(V) \smap{\cong} \Q_2(V_\top)$,
and the kernel and cokernel of $\BW(V)\to\RGB(V_\top)$ 
are the same as the kernel and cokernel of $\Br(V)\to\RBr(V_\top)$.
\end{lemma} 

\begin{proof}
Cox's Theorem \cite{Cox} states that
\[ 
H^*_\et(V,\Z/2) \cong \H^*_G(V_\top,\Z/2). 
\]
By \eqref{eq:Q_2} and Proposition \ref{ext:Q_2}, 
$\Q_2(V)\map{\cong} \Q_2(V_\top)$. 
By Theorems \ref{BW} and \ref{RBW}, 
there is a commutative diagram
\begin{align*}\xymatrix@R=1.5em{
0\ar[r]& \Br(V)\ar[d]\ar[r]&\BW(V)\ar[d]\ar[r]&\Q_2(V)\ar[d]^{\cong}\ar[r]&0
\\
0\ar[r]& \RBr(V_\top) \ar[r] &\RGB(V_\top) \ar[r] &\Q_2(V_\top) \ar[r]&0.
}\end{align*}
The result follows by the Snake Lemma.
\end{proof}

\begin{example}
If $X$ is a circle with involution, there are three cases:
\begin{itemize}\setlength{\itemindent}{-7mm}
\item $S^{2,0}$, the circle with the antipode involution.
Then $\RBr(X)=0$ by Lemma \ref{RBr:dim=1}, and
$\RGB(X)\cong\Q_2(X)\cong\Z/4$ by 
Theorem \ref{RBW} and Proposition \ref{ext:Q_2}.

\item $S^{0,2}$, the circle with trivial involution. 
By Example \ref{fixed:RGB}, 
\[  \RGB(S^{0,2})\cong \Z/8\oplus\Z/2.  \] 

\item $S^{1,1}$, the unit circle $T$ in $\C$. 
Since $T/G$ is contractible, Lemma \ref{H_G^1} implies that 
$\Q(X)\cong(\Z/2)^2$ on the 
elements of Examples \ref{Xcovers} and \ref{theta=t};
the projection $\RGB(X)\to\RGB(X^G)\cong\Z/8\oplus\Z/8$
shows that $\RGB(S^{1,1})\cong \Z/8\oplus\Z/4$. 
\end{itemize}
An interesting element of $\RBr(S^{1,1})$ is given by the
Real algebra bundle $A$ whose underlying algebra is $X\times M_2(\C)$,
but whose involution over a point $z$ with $\mathrm{Im}(z)\ge0$
is the composition of complex conjugation with conjugation by the matrix
$\left(\atop{~\cos(t/2)}{-\sin(t/2)}~ \atop{\sin(t/2)}{\cos(t/2)}\right)$
where $t$ is the angle of $z$. The fixed algebra by the involution over 
the fixed points is either $M_2(\R)$ or the skew field of quaternions.
\end{example}


\medskip\goodbreak
\section{Classification of $\RBr(X)$}\label{sec:RBr}

As pointed out in Example \ref{ex:Br}, the Brauer group 
of a quasi-projective variety $V$ is isomorphic to
${_\tors}H_\et^2(V,\Gm)$, the torsion subgroup 
of $H_\et^2(V,\Gm)$.  In this section, we prove the topological analogue,
replacing \'etale cohomology by Borel's equivariant cohomology.

\begin{theorem}\label{RBr}
$\RBr(X)$ is isomorphic to ${_\tors}\H_G^3(X,\Z(1))$,
the torsion subgroup of the equivariant cohomology group $\H_G^3(X,\Z(1))$.
\end{theorem}

The proof of Theorem \ref{RBr} is postponed until after the  technical 
Lemmas \ref{rho injects}--\ref{RBr:dim=1} below.
Combining Theorems \ref{RBW} and  \ref{RBr} with Proposition \ref{ext:Q_2},
we deduce:

\begin{corollary}\label{RGB-filt}
There is a filtration on $\RGB(X)$ with associated graded groups:
\[
\H_G^0(X,\Z/2),\quad \H_G^1(X,\Z/2) \quad\mathrm{and}\quad
 {_\tors}\H_G^3(X,\Z(1)).
\]
\end{corollary}

Using Atiyah's notation \cite{Atiyah}, 
$S^{5,0}$ denotes the 4-sphere with antipodal involution.
Since the groups in \ref{RGB-filt} only depend on the 4-skeleton
of $X$, we have:

\begin{corollary}\label{4-skel}
$\RGB(X)\cong \RGB(X^{(4)})$, where $X^{(4)}$ is the 4-skeleton of $X$.
In addition, $\RGB(X)\cong \RGB(X\times S^{5,0})$.
\end{corollary}

In preparation for the proof of Theorem \ref{RBr}, 
we do some simple calculations. Since $\pi_0(X)$ is assumed finite,
we may assume that $X/G$ is connected.

\begin{lemma}\label{rho injects}
There is a natural injective homomorphism
\[
\rho_X:\RBr(X)\to{_\tors}\H^2_G(X;U_1)\cong{_\tors}\H^3_G(X;\Z(1)).
\]
\end{lemma}

\begin{proof}
As observed by Bruno Kahn \cite[p.\,698]{Kahn},
Real vector bundles of rank $n$ on a Real space $X$ are 
classified by the equivariant cohomology set $\H^1_G(X;U_n)$ where 
$G$ acts on $U_n$ by complex conjugation.  By the Skolem-Noether theorem,
the group of automorphisms of $M_n(\C)$ is 
$PU_n = U_n/U_1 \cong SU_n/\mu_n$, so
Real Azumaya algebras of rank $n$ are classified by $\H^1_G(X;PU_n)$.
Because $U_1$ is in the center of $U_n$, 
we get exact sequences of pointed sets
\begin{equation*}\xymatrix@R=1.5em{
\H^1_G(X;SU_n) \ar[d] \ar[r] &\H^1_G(X;PU_n)
                        \ar[d]^{=}  \ar[r]^{\partial} & \H^2_G(X;\mu_n) \ar[d]\\
\H^1_G(X;U_n) \ar[r] & \H^1_G(X;PU_n) \ar[r]^{\partial}& \H^2_G(X;U_1). 
}\end{equation*}
As in Grothendieck \cite[1.4]{Dix},
it follows that the image of $\H^1_G(X;PU_n)$ in $\H^2_G(X;U_1)$ 
is $n$-torsion.
Tensoring with $\End(W)$ for a rank $r$ bundle has the effect of 
replacing $n$ by $rn$; we write $U_\oo$, $PU_\oo$ and $\mu_\infty$
for the direct limit of the groups $U_n$, $PU_n$ and $\mu_n$
as $n$ varies multiplicatively. Thus there is an exact sequence
\[
\H^1_G(X;U_\oo) \to \H^1_G(X;PU_\oo) \map{\partial} {_\tors}\H_G^2(X,U_1).
\]

An easy calculation, 
similar to \cite[1.4]{Dix} or \cite[Thm.\,8]{DonovanKaroubi},
shows that $\partial(A)\cup\partial(A')=\partial(A\oo A')$, 
and we get an injective homomorphism
$\rho_X:\RBr(X)\to{_\tors}\H^2_G(X;U_1)$, natural in $X$. 
Finally, the target is ${_\tors}\H_G^3(X,\Z(1))$ by 
Lemma \ref{Q/Z(1)}.
\end{proof}

\begin{lemma}\label{Q/Z(1)}
For all $i>0$,  
$\H_G^i(X;U_1)$ and $\H^{i+1}_G(X;\Z(1))$ are isomorphic, and 
their torsion subgroup is the image 
of $\H^i_G(X;\mu_\infty)$ in $\H_G^i(X;U_1)$.
\end{lemma}

\begin{proof}
We have a diagram of distinguished triangles in the derived category
of equivariant sheaves on $X$:
\begin{equation*}\xymatrix@R=1.5em{
\mu_\infty[-1] \ar[r] \ar[d]& 0\ar[r]\ar[d] & 
\mu_\infty \ar[r]^{\cong} \ar[d] & \mu_\infty \ar[d] \\
\Z(1) \ar[d] \ar[r] & \cO_\R \ar[r]^{\exp} \ar[d] & U_1 \ar[r] \ar[d] 
  & \Z(1)[1] \ar[d] \\
\bQ(1)\ar[r] & \cO_\R \ar[r] & U_1\oo\bQ \ar[r]^{-1} & \bQ(1)[1].
}\end{equation*}
Here `$\cO_\R$' denotes the soft sheaf of continuous sections of $\R_X$,
so $\H_G^i(X,U_1)\cong\H_G^{i+1}(X,\Z(1)$ for $i>0$.
The cohomology of the terms in the bottom row are uniquely divisible.
The result now follows from the cohomology sequences of the columns, 
such as
\[ \H_G^i(X,\bQ(1)) \smap{} \H^i_G(X;\mu_\infty) \to
   \H^{i+1}_G(X;\Z(1))  \to\H_G^{i+1}(X,\bQ(1)). \qedhere \]
\end{proof}


\begin{subremark}\label{H3=H0+H2}
If $G$ acts trivially on $X$, 
$\H_G^n(X,\Z(1))\cong H^n(X\!\times\!BG,\Z(1))$.
This equals the group hypercohomology $\H^n(G,C^*(X)\oo\Z(1))$.
Using the hypercohomology spectral sequence 
${^I}\!E_2^{pq}=H^pH^q(G,C^*(X)(1))$
of \cite[6.1.15]{WH}, and the fact that 
$\H_G^3(\pt,\Z(1))\cong\Z/2$ 
is a summand of $\H_G^n(X,\Z(1))$, we see that
$\H_G^3(X,\Z(1))\cong H^0(X,\Z/2)\oplus H^2(X,\Z/2)$.
\end{subremark}

\begin{lemma}\label{free G-cell}
If $X$ is the free bouquet of spheres $S^n(G_+)=\bigvee_G S^n$,
then $\H_G^3(X,\pt,\Z(1))$ is 0 for $n\ne3$, and $\Z$ for $n=3$.
\end{lemma}

\begin{proof}
$\H_G^3(S^n(G_+),\pt,\Z(1))\cong \H_G^{3-n}(G_+,\pt,\Z(1))
\cong H^{3-n}(\pt,\Z)$. 
\end{proof}

\begin{lemma}\label{RBr:dim=1}
When $\dim(X)\le1$, $\H_G^3(X,\Z(1))\cong H^0(X^G,\Z/2)$.
\end{lemma}

\begin{proof}
If $G$ acts freely on $X$, we have $\H_G^3(X,\Z(1))\cong H^3(X/G,\Z)=0$,
and $H^0(X^G)=0$ as well. If $G$ acts trivially on $X$, then 
$\H_G^3(X,\Z(1))\!\cong H^0(X,\Z/2)$ by Example \ref{fixed:RGB}.
Thus we may assume $X$ is connected with a fixed base point.
We need to show that $\H_G^3(X,\pt,\Z(1))=0$. 

For visual simplicity, let us write 
$\widetilde\H_G^n(Y)$ for $\H_G^n(Y,\pt,\Z(1))$ when $Y$ is pointed.
We have an exact sequence
\[
\widetilde\H_G^3(X/X^G)
\to \H_G^3(X,\Z(1))\to \H_G^3(X^G\!,\Z(1)) \to
\widetilde\H_G^3(X/X^G).
\]
The first and last terms in the display are~0, because
$X/X^G$ is a bouquet of copies of $Y=S^1\!\wedge\!(G_+)$,
and $\widetilde\H_G^3(Y)=\widetilde\H_G^4(Y)=0$.
Hence $\H_G^3(X,\Z(1))\cong \H_G^3(X^G\!,\Z/2)$.  
By Example \ref{fixed:RGB} this is $H^0(X^G\!,\Z/2)$.
\end{proof}


\smallskip
\begin{proof}[Proof of Theorem \ref{RBr}]
By Lemma \ref{rho injects}, $\rho_X$ is an injection; we need
to show it is onto.
When $G$ acts trivially, $\rho_X$ is the isomorphism
\[ \RBr(X) \cong H^0(X,\Z/2)\oplus H^2(X,\Z/2)\cong\H_G^3(X,\Z(1)) \]
of Example \ref{fixed:RGB} and Remark \ref{H3=H0+H2}.
When $X=S^n(G_+)$, the map $\rho_X$ is trivially onto,
because $\H_G^3(X,\pt,\Z(1))$ is torsion-free by Lemma \ref{free G-cell}.
Thus $\rho_X$ is a bijection for all of the test $G$-spaces 
$S^n((G/H)_+)$ of the Appendix.
We claim that $\rho_X$ is an isomorphism if $\dim(X)\le1$.
By Brown's Theorem \ref{thm:Brown}, this will imply that
$\rho_X$ is a bijection for all $X$.

So suppose that $\dim(X)=1$ and that $X^G$ has $\nu>0$ components.
By Lemma \ref{H_G^1} and Remark \ref{Xcovers}, there are Real vector
bundles $L_i$ and symmetric forms $\theta_i:L_i^{\oo2}\smap{\cong}\C_X$
on $X$ whose classes $[(L_i,\theta_i)]$ in $\Pic_G(X)\cong\Pic_G(X^G)$
map to a basis of $(\Z/2)^\nu$ under the sign map of Example \ref{ex:sign}. 
The Clifford algebras
$A_i=\Cl(L_i\oplus\C_X,\theta_i\perp1)$ have 
$\hatZ(A_i)=A_0\cong \C_X\oplus L_i$, because this is true fiberwise.
\end{proof}

\begin{proposition}\label{RGB(graph)}
When $X$ is a connected 1-dimensional $G$-complex,
and $X^G$ has $\nu$ components, we have $\RBr(X)\cong(\Z/2)^\nu$ and
\[ 
\RGB(X) = 
\begin{cases}(\Z/8)\oplus(\Z/4)^{\nu-1} \oplus H^1(X/G,\Z/2), & \nu>0, \\ 
             \Z/4\oplus \tH^1(X/G,\Z/2), &\nu=0,
\end{cases}\]
where $\tH^1(X/G,\Z/2)$ denotes the
quotient of $H^1(X/G,\Z/2)$ by the subgroup generated by 
the element $[-1]=w_1(X\!\times\!\R(1))$. 
\end{proposition}

\begin{proof}
Recall from Theorem \ref{RBW} that $\RGB(X)$ is an extension of 
$\Q_2(X)$ by $\RBr(X)$, and that $\RBr(X)\cong(\Z/2)^\nu$
by Lemma \ref{RBr:dim=1}. By Proposition \ref{ext:Q_2} and 
Lemma \ref{H_G^1}, $\Q_2(X)$ is a nontrivial extension of $\Z/2$ by 
$\Q(X)\cong (\Z/2)^\nu\oplus H^1(X/G,\Z/2)$.
The case $\nu=0$ is now immediate.

If $\nu>0$ then $\Q_2(X)\cong \Z/4\oplus(\Z/2)^{\nu-1}\oplus H^1(X/G,\Z/2)$
by Example \ref{ex:sign}, Lemma \ref{H_G^1} and
Proposition \ref{ext:Q_2}. The projection 
$\RGB(X)\to\RGB(X^G)\to\oplus_\nu\RGB(\pt)\cong(\Z/8)^\nu$ 
shows that the extension $\RGB(X)$ of $\Q_2(X)$ by $\RBr(X)$ 
is as described.
\end{proof}

\begin{subremark}
When $X$ is connected, $\dim(X)=1$ and $\nu>0$ we have 
\[
WR(X)\cong 
\Z^{\nu}\oplus H^1(X/G,\Z/2),
\]
and $WR(X)\to\RGB(X)$ is onto.
Since the proof requires slightly more machinery, 
we will prove this in \cite[3.1]{KW}.
\end{subremark}

\begin{example}
Let $X$ be a compact connected oriented 2-manifold of genus $g$.
If $G$ acts freely on $X$, then
\[ 
\RGB(X) \cong \Z/4\oplus (\Z/2)^g.
\]
Indeed, we saw in \cite[4.6]{KSW} that
$H^1(X/G,\Z/2)\cong(\Z/2)^{g+1}$; by Lemma \ref{H_G^1} this is
$\H_G^1(X,\Z/2)$. Since $\H_G^3(X)=H^3(X/G)=0$, 
$\RGB(X)\cong\Q_2(X)$
and the result follows from Proposition \ref{ext:Q_2}.

If $G$ does not act freely, and
$X^G$ is the union of $\nu>0$ circles, similar calculations 
(which we omit) show that $\RBr(X)=(\Z/2)^\nu$ and
$\RGB(X)\cong \Z/8\oplus(\Z/4)^{\nu-1}\oplus(\Z/2)^g$.
\end{example}

\begin{proposition}\label{RGB:curves}
Let $V$ be a smooth projective curve of genus $g$, 
defined over $\R$ (and geometrically irreducible).
If $V(\R)$ has $\nu$ components, then
\[
\BW(V) \map{\cong}\RGB(X) \cong \begin{cases}
\Z/4 \oplus (\Z/2)^g, & \nu=0;\\
\Z/8 \oplus (\Z/4)^{\nu-1} \oplus(\Z/2)^g, & \nu>0.
\end{cases}
\]
We also have $\BW(V) \map{\cong}\RGB(V_\top)$ when
$V$ is a smooth affine curve defined over $\R$.
\end{proposition}

\begin{proof}
Set $X=V_\top$.  We know that $H^1_\et(V,\Z/2)\cong (\Z/2)^{g+1+s}$,
where $s=0$ if $\nu=0$ and $s={\nu-1}$ if $\nu>0$; see \cite[0.5]{PW}.
By Proposition \ref{ext:Q_2}, we have
$\Q_2(V)\cong\Q_2(X)\cong \Z/4 \oplus (\Z/2)^{g+s}$.
By Theorem \ref{RBr}, it suffices to observe that $\Br(V)=(\Z/2)^\nu$
(see \cite[0.1]{PW}) and $\RBr(V_\top)=(\Z/2)^\nu$ as well.
\end{proof}

\section{Stiefel--Whitney classes}\label{sec:WR}

\noindent
Following Atiyah and Segal \cite{AS}, we recall that vector bundles on $\XG$
may be identified with equivariant vector bundles on $X\!\times\!EG$;
the pullback of a bundle on $\XG$ to $X\!\times\! EG$
is an equivariant vector bundle.
Thus the map $X\!\times\!EG\to X$ induces natural ``Atiyah--Segal'' maps
$KO_G(X)\to KO(\XG)$, where $KO(\XG)$ is defined to be
representable $K$-theory.

\begin{defn}\label{def:w_n}
The equivariant Stiefel-Whitney classes 
$$w_n:KO_G(X)\to \H^n_G(X,\Z/2)$$
are the composition of the Atiyah-Segal map 
with the usual Stiefel-Whitney classes
$w_n:KO(\XG)\to H^n(\XG,\Z/2)\cong \H^n_G(X,\Z/2)$.
\end{defn}

When $X=G\times Y$, for example, 
we have $KO_G(X)\cong KO(Y)$ and $\H^n_G(X)\cong H^n(Y)$.
In this case, the $w_n$ are the usual Stiefel-Whitney classes
$w_n: KO(Y)\to H^n(Y,\Z/2)$.

\begin{subremark}\label{Chern}
In \cite{Kahn}, Bruno Kahn defined equivariant Chern classes 
$c_n:KR(X)\to \H^{2n}_G(X,\Z(n))$ for Real vector bundles, with
the first Chern class $c_1$ inducing an isomorphism between the group of
rank~1 Real vector bundles on $X$ and $\H_G^2(X,\Z(1))$.
($\Z(1)$ is the sign representation of $G$.) 
In particular, $c_1:KR(X)\to \H_G^2(X,\Z(1))$ is a surjection.
\end{subremark}



In the algebraic setting, the discriminant does not factor 
through $W(V)$, because the discriminant of $h(E)$ is $(-1)^{\rank E}$;
instead, it factors through the ideal $I(V)$. The same is true
in our setting: $w_1$ does not factor through $WR(X)$;
the case $X=\pt$ shows that the composition 
$KR(X)\map{h}KO_G(X)\map{w_1}\H_G^1(X,\Z/2)$
need not be zero.

To see that $w_1$ factors through the ideal $I(X)$
of forms in $WR(X)$ of even degree, let $\hat{I}(X)$ denote the 
kernel of $\rank:KO_G(X)\to\Z$.
The quotient map $KO_G(X)\to WR(X)$ sends $\hat{I}(X)$ onto the
ideal $I(X)$ of $WR(X)$ because if $E$ has rank $2n$ then
$[E]-[h(n)]$ has rank~0.
By Proposition \ref{w1-bundles} and Lemma \ref{WR/I2}, the map 
\[
WR(X)\ \map{\Cl}\ \RGB(X)\map{\hatZ} \Q_2(X)
\]
sends $I(X)$ of $WR(X)$ to the subgroup $\H_G^1(X,\Z/2)$ of $\Q_2(X)$.

\begin{lemma}\label{w1 on I}
The composition 
\[
\hat{I}(X)\to I(X)\subset WR(X) \map{\Cl}\RGB(X)\map{\hatZ}\ \Q_2(X)
\]
agrees with 
the Stiefel--Whitney class $w_1:KO_G(X)\to\H_G^1(X,\Z/2)$.
Hence it induces a
Stiefel--Whitney class $w_1:I(X)\to\H_G^1(X,\Z/2)$.
\end{lemma}

\begin{proof}
By construction, the map $w_1:KO_G(X)\to\H_G^1(X,\Z/2)$
is the composition of the determinant map $KO_G(X)\to\Pic_G(X)$
with the isomorphism between $\Pic_G(X)\cong\Pic_G(X\times EG)$
and $\H_G^1(X,\Z/2)$. Given an equivariant $\R$-linear bundle $F$
of even rank $d$, corresponding to the symmetric form $q$ on the
Real bundle $E=F\oo\C$,
let $\cA$ denote the Clifford algebra $\Cl(q)$; the determinant bundle
$\wedge^dE$ is a summand of $\cA_0$. A direct calculation 
shows that $\hatZ(\cA)_x=\C\oplus\wedge^dE_x$ on each fiber,
and hence that $\hatZ(\cA)=\wedge^dE$, as asserted.
\end{proof}

\begin{theorem}\label{w1:factors}
The algebraic discriminant of a smooth variety $V$ factors as
\[
I(V) \to I(V_\top)\ \map{w_1}\ \H^1_G(V_\top,\Z/2)\cong H^1_\et(V,\Z/2).
\]
\end{theorem}

\begin{proof} 
It suffices to consider elements of $W(V)$ of the form
$u=(E,\theta)-(\cO_V^n,1)$, where $\theta$ is a symmetric form on 
an algebraic vector bundle $E$ of rank $n$. 
Since $\disc(u)=\disc(\det E,\det\theta)$, we may
replace $E$ by $\det E$ to assume that $E$ has rank~1.  
Note that $\disc(\cO_V,1)$ is trivial.

By restricting $V$ to an open subvariety $U$, 
we may assume that $E$ is trivial and $\theta$ is a global unit.
This doesn't affect the discriminant, 
as $H^1_\et(V,\Z/2)\cong H^0(V,\cH^1)$ is a subgroup of 
$H^1_\et(U,\Z/2)$, by the Bloch--Ogus sequence
\[
0\to H^0(V,\cH^n) \to H^n_\et(k(V)) \to \oplus_x\ H^{n-1}_\et(k(x)),
\]
and the sequence of
Theorem \ref{w1:factors} is natural in $V$.
Since the discriminant is a homomorphism, 
we are reduced to the case when $E$ is the
trivial bundle $\cO_V$ and $\theta$ is given by a global
unit $a$ of $H^0(V,\cO_V)\subset F$.  
By construction, the algebraic discriminant sends 
$[\cO_V,a]$ to the class of $a\in F^\times/F^{\times2}$.  
Since the map $W(V)\to WR(V_\top)$ sends the class of $(\cO_V,a)$ 
to the class $(V_\top\!\times\!\C,a)$, we need to evaluate $w_1$
on forms $(E,a)$, $E=V_\top\!\times\!\C$.

By Example \ref{theta=t}, the trivial Real line bundle $\C_T$ on
the unit circle $T$ in $\C$ carries a canonical
Real symmetric form $\theta=t$, which is multiplication by $t$ on 
the fiber over $t\in T$, and $w_1(\C_T,t)$ is nontrivial in
$\H_G^1(T,\Z/2)\cong\{\pm1\}^2$.
If $V=\Spec(A)$, $A=\R[t,1/t])$, then $V_\top=\C-\{0\}\simeq T$ and
under the isomorphism 
$\H_G^1(T,\Z/2)\cong H_\et^1(V,\Z/2) \cong A^\times/A^{\times2}$,
$w_1(\C_T,t)$ is the class of the unit $t$ of $A$.

A global unit $a$ of a variety $V$ over $\R$ defines an 
equivariant map $V_\top\smap{a} \C-\{0\}\simeq T$, and 
$a^*: KO_G(T)\to KO_G(V_\top)$ sends $(T_\C,t)$ to $(E,a)$.
By naturality,
\[
w_1(E,a) = a^*w_1(T,t) = a^*[t]=[a].
\qedhere\]
\end{proof}

We remark that Theorem \ref{w1:factors} is well known in the affine
case; see \cite[V.2.5]{Lam}. Our argument uses $\hat{I}(V)$
to avoid the cases of the signed determinant that arise in loc.\,cit.\ 



\medskip\goodbreak
\begin{center}{\it The class $w_2$}\end{center}
\smallskip

Recall that rank~2 real bundles on $Y$ are classified by
$H^1(Y,O_2)$, and that the orthogonal group $O_2$ is the 
semidirect product $S^1\rtimes O_1$, where $O_1\cong\Z/2$ 
is the diagonal subgroup $\mathrm{diag}(\pm1,1)$ of $O_2$. 
($O_1$ acts by complex conjugation).  
Multiplication by~2 on $S^1$ extends to an 
endomorphism $q$ of $O_2$ fixing $O_1$.
Thus there is an exact sequence
\begin{equation}\label{eq:O2}
1 \to \Z/2 ~\map{(\pm1,1)}~ O_2 ~\map{q}~ O_2 \to 1.
\end{equation}
The boundary map $H^1(Y,O_2) \map{\partial} H^2(Y,\Z/2)$
gives an invariant of rank~2 real bundles.

\begin{lemma}\label{w_2}
Let $E$ be an $\R$-linear vector bundle on $Y$ of rank~2, classified by
$\xi\in H^1(Y,O_2)$. Then the element $\partial(\xi)$ of
$H^2(Y,\Z/2)$ is the Stiefel--Whitney class $w_2(E)$.
\end{lemma}
 
\begin{proof}(Folklore)
Since $E$ is the pullback of $EO_2$, the universal bundle on $BO_2$,
we may assume that $Y=BO_2$ and $E=EO_2$. 
Now the vector space $H^2(BO_2,\Z/2)$
is 2-dimensional, with basis $\{w_1^2, w_2\}$, so we can write
$\partial(\xi)$ as $aw_1^2(E)+bw_2(E)$. The restriction $E'$ of $E$ to $BSO_2$
is an oriented bundle satisfying $w_1^2(E')=0$ and $w_2(E')\ne0$ 
\cite[12.4]{MS}, so $\partial(\xi)$ restricts to $w_2(E')$ in
$H^2(BSO_2,\Z/2)\cong\Z/2$; hence $b=1$. On the other hand, 
the restriction of $EO_2$ to $BO_1=\R\bP^\infty$ is $L\oplus1$, 
and its pullback by $q$ is trivial.
Since $w_1^2(L)\ne0$, $w_2(L)=0$ we see that $a=0$.
\end{proof}

\goodbreak
\begin{proposition}\label{WR-w2}
The reduction modulo~2 of the equivariant Chern class $c_1$ of a
Real vector bundle is the equivariant Stiefel-Whitney class $w_2$
of its underlying $\R$-linear bundle.
That is, the left square commutes in the diagram: 
\begin{align*}\xymatrix{
KR(X) \ar[d]^{c_1} \ar[r]^h & KO_G(X) \ar[d]^{w_2}\ar[r]^{\mathrm{onto}}
& WR(X)\ar@{-->}[d]^{\bar{w}_2} \ar[r] &0 \\
\H^2_G(X,\Z(1))\ar[r]  & \H^2_G(X,\Z/2)\ar[r]^{\tilde\beta} & 
{_2}\H^3_G(X,\Z(1))\ar[r] &0.
}\end{align*}
Hence the map $w_2:KO_G(X)\to \H^2_G(X,\Z/2)$ induces a homomorphism
\[ \bar{w}_2:WR(X)\to {_2}\H^3_G(X,\Z(1)). \]
\end{proposition}

\begin{proof}
As pointed out by Atiyah (in the proof of Theorem 2.5 in \cite{Atiyah}),
there is a Splitting Principle for Real vector bundles.
As a consequence, 
it is enough to consider the case of a Real line bundle $L$.
Consider the diagram of equivariant sheaves of groups,
whose second row is \eqref{eq:O2}, 
and where $\exp(t)=e^{2\pi it}$:
\begin{align*}\xymatrix@R=1.5em{
1\ar[r]& \Z(1)\ar[r]\ar[d] &\R(1)\ar[r]^{\exp}\ar[d]& S^1 \ar[r]\ar[d]^{j} &1 \\
1\ar[r]& \Z/2  \ar[r]^{(\pm 1,1)} & O_2\ar[r]^{q}    & O_2 \ar[r] &1.
}\end{align*}
The right vertical map $j$ is the standard inclusion, and the middle vertical
is the map $t\mapsto\exp(t/2)$. The Real line bundle $L$ determines 
an element $[L]$ in $\H^1_G(X,\C^\times)\cong \H^1_G(X,S^1)$,
and the Chern class $c_1(L)$ in $\H^2_G(X,\Z(1))$ is the coboundary 
$\partial([L])$; see \cite[Prop.\,1]{Kahn}.
Now $j[L]\in \H^1_G(X,O_2)$ is the class of the underlying $\R$-linear
$G$-bundle $\xi$. 
By naturality and Lemma \ref{w_2} applied to $Y=X\times_GEG$, 
the reduction modulo~2 of $c_1(L)$ is $\partial(j[L])=w_2(\xi)$. 
\end{proof}

\begin{subremark}
Proposition \ref{WR-w2} is analogous to the classical fact \cite[14-B]{MS}
that for a complex vector bundle $E$, the reduction modulo~2 of
the usual Chern class $c_1(E)$ is $w_2(E)$.
When $X=X^G$, the proof of \ref{WR-w2} 
is due to B.\,Kahn \cite[Thm.\,4]{Kahn}.
\end{subremark}

\goodbreak

Recall from Remark \ref{RGB-RBr} that 
forgetting the grading yields a function $u:\RGB(X)\to\RBr(X)$.

\begin{lemma}\label{u=w_2}
The ungraded Clifford algebra map 
\[
WR(X)\map{C} \RGB(X) \map{u} \RBr(X) \cong {_\tors}\H_G^3(X,\Z(1))
\]
is $\bar{w}_2 = \tilde\beta\;w_2$.
Thus  the image of $u\!\circ\!\Cl$ in $\RBr(X)$ has exponent~2.
\end{lemma}
\goodbreak

\begin{proof}
Recall that $\H_G^3(X,\Z(1))\cong\H_G^3(X\times S^{5,0},\Z(1))$,
where $S^{5,0}$ is the 4-sphere with antipodal involution.
Thus we may replace $X$ by $X\times S^{5,0}$ to assume that $G$
acts freely on $X$. In this case, $KO_G(X)\cong KO(X/G)$
and we noted in Example \ref{fixed:RGB} and Theorem \ref{RBr},
\[
\H_G^3(X/G,\Z(1))\cong H^0(X/G,\Z/2)\oplus H^2(X/G,\Z/2).
\]

Suppose then that $G$ acts freely on $X$. 
As noted in Remark \ref{GR=KOG}, $WR(X)$ is a quotient of 
$KO_G(X)\cong KO(X/G)$. If $F$ is an $\R$-linear vector bundle
on $X/G$, Donovan and Karoubi proved in \cite[Lemma 7]{DonovanKaroubi}
that $w_2(F)\in H^2(X/G,\Z/2)\cong\H_G^2(X,\Z/2)$ 
coincides with the class of the ungraded $\R$-linear Clifford algebra
$u(\Cl_{\R}(F))$ on $X/G$ or, equivalently, with the class of
$u(\Cl_{\R}(F)\oo\C)=u\Cl_{\C}(E,q)$ on $X$. 

The isomorphism $KO_G(X)\cong GR(X)$ of Remark \ref{GR=KOG} 
sends the class of $F$ to the class of the Real vector bundle 
$E=F\oo_{\R}\C$ on $X$ with form $q$. From Proposition \ref{WR-w2}
we see that $\bar{w}_2(E,q)=\tilde\beta\;w_2(F)$.
Since $w_2(F)=u\Cl_{\R}(F)$, we are done.
\end{proof}

\smallskip\goodbreak

\begin{lemma}\label{Hasse=w2}
The algebraic Hasse invariant on a smooth variety $V$ 
is compatible with the equivariant $\bar{w}_2$ on $X=V_\top$ 
in the sense that the following diagram commutes:
\[
\xymatrix@R=2.0em{ 
K(V) \ar[r]^h\ar[d] & GW(V) \ar[r]^-{\mathrm{Hasse}} \ar[d] & 
{_2}\Br(V) \ar[d]^{\cong}\ar[r]^-{\cong} & 
H^0(V,\cH^2) \ar[d]_{\mathrm{onto}} \\
KR(X)\!\ar[r]^h & GR(X) \ar[r]^-{w_2}    
&     \H^2_G(X,\Z/2)/(\Pic V/2) \ar[r]^-{\mathrm{onto}}
& {_2}\H^3_G(X,\Z(1)).\kern-5pt
}\]
\end{lemma}

\begin{proof}
We showed in \cite[1.2]{KSW} that the left square commutes as 
a special case of a general result about Hermitian categories.
The lower right map is well defined because the map
$\Pic(V)\to \H^2_G(X,\Z/2)$  factors through the Chern class 
$c_1:KR(X)\to \H^2_G(X,\Z(1))$ of Remark \ref{Chern}, 
and there is a short exact sequence:
\[
0\to \H^2_G(X,\Z(1))/2 \to \H^2_G(X,\Z/2)\to{_2}\H^3_G(X,\Z(1))\to0.
\]
Commutativity of the right square is immediate from this.
Thus it suffices to show that the middle square commutes.

Let $\theta$ be a symmetric form on an algebraic vector bundle $E$.
As in the proof of \ref{w1:factors}, 
we may restrict $V$ to any dense open subvariety $U$, because
$H^0(V,\cH^2)\to H^0(U,\cH^2)$ is an injection;
both are subgroups of $H^2_\et(F,\Z/2)$ by Bloch--Ogus \cite{BlochOgus}.
Over the function field $F$, $(E,\theta)$ is isomorphic to a sum of 
1-dimensional forms. Such an isomorphism is defined over a 
dense open $U$. Replacing $V$ by such a $U$, we may assume that
$(E,\theta)$ is a Whitney sum of rank~1 forms $(E_i,\theta_i)$.
By construction, 
the algebraic Hasse invariant of $(E,\theta)$ is
the class of $\prod_{i<j}\disc(E_i,\theta_i)\cup \disc(E_j,\theta_j)$.
By Theorem \ref{w1:factors}, this equals
$w_2(E,\theta) = \prod_{i<j}w_1(E_i,\theta_i)\cup w_1(E_j,\theta_j)$.
\end{proof}


It is not generally the case that every element of $H^2(X,\Z/2)$
is the second Stiefel--Whitney class of a vector bundle
(see Example \ref{H^5(X,Z)} below).
The following criterion was communicated to us by J.\ Lannes.
Let $\beta$ denote the integral Bockstein $K(\Z/2,4)\to K(\Z,5)$,
representing the cohomology operation $H^4(-,\Z/2)\to H^5(-,\Z)$.

\goodbreak
\begin{theorem}[Lannes]\label{u=w_2(E)}
Let $u$ be a cohomology class in $H^2(X,\Z/2)$.
A necessary condition for $u$ to be $w_2(E)$
for an $\R$-linear vector bundle $E$ on $X$ is that $\beta(u^2)=0$.

If $\dim(X)\!\le\!7$ and $\beta(u^2)\!=\!0$ in $H^5(X,\Z)$,
then there is a vector bundle $E$ such that $u=w_2(E)$.
This is always the case if $\dim(X)\le4$.
\end{theorem}

\begin{proof}
Let $P$ denote the homotopy pullback of the Steenrod square
$\Sq^2\!: K(\Z/2,2)\to K(\Z/2,4)$ along the reduction
$K(\Z,4)\to\! K(\Z/2,4)$. The Stiefel--Whitney map $w_2:BSO\to K(\Z/2,2)$
and the Pontrjagin class $p_1:BSO\to K(\Z,4)$ are compatible 
(see \cite[15-A]{MS}), so (up to homotopy) 
they factor through a map $s:BSO\to P$.  
Thus we have the diagram:
%
\begin{align*}\xymatrix@R=1.5em{
BSO\ar[r]^{s}& P\ar[r]^-{}\ar[d]^{p_1} &K(\Z/2,2)\ar[d]^{\Sq^2} \\
& K(\Z,4)\ar[r]& K(\Z/2,4) \ar[r]^{\beta} & K(\Z,5).
}\end{align*}
As the composition $BSO\to K(\Z,5)$ is null-homotopic, 
a necessary condition
for $u\in H^2(X,\Z/2)$ to be $w_2(E)$ is that $\beta(u^2)=0$.

Now suppose that $\beta(u^2)=0$. Since the bottom sequence is a
homotopy fibration sequence,
the map $u^2:X\to K(\Z/2,4)$ lifts to a map $X\to K(\Z,4)$
and hence to a map $X\to P$. If $\dim(X)\le7$, 
this lifts to a map $f:X\to BSO$, because
the map $s$ is 7-connected.
The classifying map $f$ determines a stable vector bundle $E$ on $X$
with $w_2(E)=u$.
\end{proof}


\begin{subremark}\label{dim=3}
If $\dim(X)\le3$ then $w_2$ maps the kernel of 
$(\rank,\det):KO(X)\to H^0(X,\Z)\oplus H^1(X,\Z/2)$ isomorphically
onto $H^2(X,\Z/2)$. This follows from the proof of Theorem \ref{u=w_2(E)}
but it is also a consequence of the Atiyah--Hirzebruch spectral sequence.
\end{subremark}

\begin{example}\label{H^5(X,Z)}
Cartan proved in \cite{Cartan} that $\beta\,\Sq^2(u)\!=\!\beta(u^2)$ 
is a nonzero cohomology operation $H^2(X,\Z/2)\to H^5(X,\Z)$. 
Equivalently, the universal class
$u\in H^2(K(\Z/2,2),\Z/2)$ has $\beta(u^2)\ne0$.
This is easily proven, using universal coefficients and the 
well known cohomology groups $H^*(K(\Z/2,2),\Z/2)$.

If $X$ is the 6-skeleton of $K(\Z/2,2)$, the restriction 
$\bar{u}\in H^2(X,\Z/2)$ of $u$ cannot be $w_2(E)$ for any 
$\R$-linear vector bundle $E$ on $X$, as $\beta(\bar{u}^2)\ne0$.
\end{example}


\begin{theorem}\label{WR-RGB}
If $\dim(X)\le4$, and $G$ acts freely on $X$,
the image of $WR(X)\smap{\Cl} \RGB(X)$ 
fits into an extension:
\[ 0 \to  {_2}\!\RBr(X) \to  \Cl(WR(X)) \map{\hatZ} \Q_2(X) \to 0.  \]
\end{theorem}

\begin{proof}
Recall from Lemma \ref{WR/I2} that 
$WR(X)\smap{\Cl} \RGB(X)\smap{\hatZ}\Q_2(X)$ is onto.
Let us write $I_2(X)$ for the kernel of this map.
By Lemma \ref{u=w_2}, 
we have a commmutative diagram with exact rows:
\begin{align*}\xymatrix@R=1.5em{
0\ar[r]& I_2(X)\ar[r]\ar[d]^{\bar{w}_2}& WR(X)\ar[r]^-{}\ar[d]^{\Cl}
    &\Q_2(X)\ar[d]^{=}\ar[r] & 0 \\
0\ar[r]& \RBr(X)\ar[r]& \RGB(X) \ar[r]^{} & \Q_2(X)\ar[r]& 0.
}\end{align*}
By Theorem \ref{u=w_2(E)}, 
$KO_G(X)\cong KO(Y)\smap{w_2} H^2(Y,\Z/2)\cong\H_G^2(X,\Z/2)$ is onto,
where $Y=X/G$.
By Proposition \ref{WR-w2}, $w_2$ induces a surjection 
$\bar{w}_2:I_2(X)\to  {_2}\H_G^3(X,\Z(1))\cong {_2}\!\RBr(X)$.
\end{proof}

\begin{subremark}
We do not think the assumption that $G$ acts freely on $X$ is
necessary in Theorem \ref{WR-RGB}.  For example,
when $G$ acts trivially on $X$ 
then $KO_G(X)\cong KO(X)\oo R(G)$ and 
$w_2:KO_G(X)\to KO(X)\to H^2(X,\Z/2)$ is onto by Theorem \ref{u=w_2(E)}.
In addition, the map $\tilde\beta:H^2(X,\Z/2)\to {_2}\H_G^3$ is a surjection.
Thus the proof of \ref{WR-RGB} goes through, using Example \ref{fixed:RGB}.
\end{subremark}

\section{Brauer--Wall versus $\RGB(X)$}\label{sec:BW-RGB}

In this section, we compare the groups $\BW(V)$ and $\RGB(V_\top)$
for smooth varieties defined over $\R$.

First, suppose that $V$ is a complex variety, with 
underlying topological space $V(\C)$.
In this case, the topological space $V_\top$ with involution
associated to $V$ is $G\times V(\C)$, because
$V\otimes_{\R}\C$ is two copies of $V$. 
From Proposition \ref{free:RGB} 
we see that
\begin{equation}\label{eq:RGB(V)}
\RGB(V_\top) \cong H^0(V(\C),\Z/2)
\oplus H^1(V(\C),\Z/2)\oplus{_\tors}H^3(V(\C),\Z).
\end{equation}
This calculation of $\RGB(V_\top)$ also follows from 
Theorems \ref{RBW} and \ref{RBr}.

If $V$ is smooth and projective over $\C$, $H^{p+q}(V(\C),\C)$
has a Hodge decomposition as the sum of $H^p(X,\Omega^q)$, 
and we write $h^{q,p}$ for the dimension of $H^p(X,\Omega^q)$.
In particular $h^{0,2}$ is the dimension of 
$H^2(V,\cO_V)\cong H^2_\an(V,\cO_\an).$
Abelian varieties of dimension $n$ have $h^{0,2}=\binom{n}2$, while
projective spaces and ruled surfaces have $h^{0,2}=0$. 

Let $\rho$ denote the rank of the natural map
$H^2_\an(V,\Z(1))\!\to\! H^2_\an(V,\cO_\an)$.

\begin{theorem}\label{h02=0}
Suppose that $V$ is a smooth complex projective variety. Then 
\begin{enumerate}
\item $Br(V)\cong (\bQ/\Z)^\rho \oplus\ {_\tors}H^3(V(\C),\Z)$ 
\item there is a split exact sequence
\[
0\to (\bQ/\Z)^{\rho} \to \BW(V)\to \RGB(V_\top) \to 0,
\]
where $\rho\ge2\,h^{0,2}$, and $\rho=0$ when $h^{0,2}=0$.
\item $\BW(V)\to\RGB(V_\top)$ 
is an isomorphism if and only if $h^{0,2}=0$.
\end{enumerate}
\end{theorem}

\begin{proof}
Recall \cite{deJong} that the Brauer group
is the torsion subgroup of $H^2(V,\Gm)$,
and that $H^2_\an(V,\cO_\an)$ and $H^3_\an(V,\cO_\an)$ are torsionfree.
Consider the exponential sequence, 
in which we have identified $\Z$ and $\Z(1)$:
\[ H^2_\an(V,\Z) \smap{\eta} H^2_\an(V,\cO_\an) \to
H^2(V,\Gm) \to H^3_\an(V,\Z)\to H^3_\an(V,\cO_\an).
\]
If $h^{0,2}=0$, the second term is zero, so the torsion subgroups of
$H^2(V,\Gm)$ and $H_\an^3(V,\Z)\cong H^3(V(\C),\Z)$ are isomorphic.

Now suppose that $h^{0,2}(V)>0$, and write $D$ for the cokernel of $\eta$;
since $D$ is divisible, it is a summand of $H^2(V,\Gm)$. 
Thus 
\[
\Br(V) \cong D \oplus\ {_\tors}H^3(V(\C),\Z).
\]
By Lemma \ref{Cox-H1}, it remains to determine $D$.
The kernel of the map $\eta$ in the displayed exponential sequence is 
the image of $\Pic(V)=H^1(V,\Gm)$ in $H^2_\an(V,\Z)$; this is the
N\'eron-Severi group $NS(V)$; see \cite[p.\,447]{Hart}.
Since $H^2_\an(V,\cO_\an)$ is torsionfree, $NS(V)$ contains
the torsion subgroup of $H_\an^2(V,\Z)$; let $N$ denote $NS(V)$ modulo
torsion. 

By universal coefficients, $H_\an^2(V,\C)\cong H_\an^2(V,\Z)\otimes\C$,
so $H_\an^2(V,\Z)$ has rank $h^{1,1}+2h^{0,2}$.  
Since the  rank of $N$ is $\le h^{1,1}$\!, the image 
$\overline{H}^{\, 2}_\an = H_\an^2(V,\Z)/N$
of $H_\an^2(V,\Z)$ in $H_\an^2(V,\cO_\an)$ is a free abelian group of
rank $\rho$, $\rho\ge 2h^{0,2}>0$.
Since $\overline{H}^{\, 2}_\an$ 
is the image of $\eta$, the torsion subgroup of
$D$ is the nonzero group $(\bQ/\Z)^\rho$, as required.
\end{proof}


Next, we suppose that $V$ is geometrically connected over $\R$, i.e., that
$V\otimes_{\R}\C$ is connected.  In this case, $X=V_\top$ is connected.

Recall that $h^{0,2}(V)=\dim_\R H^2(V,\cO_V)$ equals $h^{0,2}(V\otimes_\R\C)$, 
because (by \cite{GAGA}),
$H^2(V,\cO_V)\cong H^2(V\otimes\C,\cO_{V\otimes\C})^G\cong H^2_\an(X,\cO_\an)^G$.
As observed by Krasnov in \cite[proof of 1.3]{Krasnov-Cohom},
this group is also isomorphic to $\H^2_G(X,\cO_\an)$ 
(because $\cO_\an$ is a torsionfree $G$-sheaf). 
The proof of the following result of Krasnov \cite[(0.4)]{Krasnov-Cohom} 
was communicated to us by Olivier Wittenberg. 

\goodbreak
\begin{theorem}\label{pg>0}
Let $V$ be a smooth geometrically connected projective variety 
over $\R$, with underlying $G$-space $V_\top$. Then
\[
\Br(V) \cong 
(\bQ/\Z)^{\rho_0} \oplus\ 
{_\tors}\H^3_G(V_\top,\Z(1)),
\]
where $\rho_0\ge h^{0,2}(V)$.  If $h^{0,2}=0$ then $\rho_0=0$.
\end{theorem}

\begin{subremark}
Recall from Theorem \ref{RBr} that ${_\tors}\H^3_G(V_\top,\Z(1))$
is isomorphic to the finite group $\RBr(V_\top)$. 
Thus Theorem \ref{pg>0} implies that
\begin{equation*}
\Br(V) \cong  \RBr(V_\top) \oplus (\bQ/\Z)^{\rho_0}.
\end{equation*}
\end{subremark}

\begin{proof}[Proof of \ref{pg>0}]
The exponential sequence 
$0\!\to\!\Z(1)\!\to\!\cO_\an\!\to\!\cO_\an^\times\!\to\!1$
of analytic $G$-sheaves on the $G$-space $X=V_\top$ 
yields an exact sequence
\[
\H^2_G(X,\Z(1))\!\smap{\eta}\! \H^2_G(X,\cO)\smap{}\!
\H^2_G(X,\Gm)\!\to\!\H^3_G(X,\Z(1))\!\to\!\H^3_G(X,\cO)
\]
where we have abbreviated $\cO_\an$ as $\cO$ for visual simplicity.
Note that the groups $\H^n_G(X,\cO)$ are torsionfree and divisible.
By Krasnov \cite[0.1]{Krasnov-Cohom}, the Brauer group
$\Br(V)$ is the torsion subgroup of $\H^2_G(X,\Gm)$.
If $h^{0,2}=0$ then the conclusion that 
$\Br(V)\cong\ {_\tors}\H^3_G(X,\Z(1))$ is immediate.

Let $\rho_0$ denote the $\Z$-rank of the image of the map $\eta$.
As observed by Krasnov \cite[0.2]{Krasnov-Cohom},
it follows that the torsion subgroup $\Br(V)$ of $\H^2_G(X,\Gm)$ 
is the direct sum of $(\bQ/\Z)^{\rho_0}$ and ${_\tors}\H^3_G(X,\Z(1))$.
Finally, Lemma \ref{H2-R(1)} below shows that $\rho_0\ge h^{0,2}$. 
\end{proof}

Let $\tau$ be the involution on $H^2(X,\C)$ 
given by complex conjugation on the coefficients;
the $\tau$-eigenspace decomposition is
$H^2(X,\C)\cong H^2(X,\R)\oplus H^2(X,\R(1))$.
By universal coefficients, $H^2(X,\Z(1))\otimes\R \cong H^2(X,\R(1)).$
When the group $G$ acts on both $X$ and the coefficients, we
have the de\,Rham action $\sigma$; it commutes with $\tau$ and we have
\[
\H_G^2(X,\Z(1))\otimes\R \cong  H^2(X,\Z(1))^G\otimes\R \cong
 H^2(X,\R(1))^G \cong \H_G^2(X,\R(1)).
\]

\begin{lemma}\label{H2-R(1)}
The map
$H^2(X,\R(1))^G\!\to\! H^2(X,\cO_\an)^G\!\cong\R^{h^{0,2}}$\!
is onto. Hence the image of 
$\H_G^2(X,\Z(1)) \smap{\eta} \H_G^2(X,\cO_\an)$
has rank  $\rho_0\ge h^{0,2}$.
\end{lemma}

\begin{proof}
Write $H^{0,2}$ for $H^2(X,\cO_\an)$, and
define $f:H^{0,2}\to H^{2}(X,\C)$ by $f(x)=(x-\tau x)/2$;
since $\tau(f(x))=-f(x)$, $f(x)$ is in $H^2(X,\R(1))$. 
As 
$\sigma$ commutes with $\tau$,
$f$ sends $(H^{0,2})^G$ to $H^2(X,\R(1))^G$.
Since $\eta(f(x))=x$, the map 
\[
\H_G^2(X,\Z(1))\otimes\R\cong H^2(X,\R(1))^G \to 
(H^{0,2})^G \cong\H_G^2(X,\cO_\an)
\] 
is onto. Hence the image $I$ of $H^2(X,\Z(1))$
spans the $h^{0,2}$-dimensional $\R$-vector space $\H_G^2(X,\cO_\an)$.
Thus the rank $\rho_0$ of $I$ is  least $h^{0,2}$.
\end{proof}

\begin{remark}
Krasnov defines the {\it Lefschetz number} of $V$ to be $\rho_0$,
the $\Z$-rank of the image of $\eta:\H^2_G(X,\Z(1))\to\H^2_G(X,\cO_\an)$.
%
Krasnov also showed in \cite[(0.6)]{Krasnov-affine} that 
Theorem \ref{pg>0} holds for any smooth quasi-projective surface 
defined over $\R$. 
However, if $h^{0,2}=0$ we may have $\rho_0>0$.
\end{remark}

\begin{theorem}\label{real:pg=0}
Let $V$ be a smooth quasi-projective variety over $\R$.
If $V$ is geometrically connected, there is an exact sequence
\[
0\to (\bQ/\Z)^{\rho_0} \to \BW(V)\to \RGB(V_\top) \to 0.
\]
If $V$ is projective, then $\rho_0\ge h^{0,2}(V)$, and 
$\BW(V)\to \RGB(V_\top)$ is an 
isomorphism if and only if $h^{0,2}(V)=0$.
\end{theorem}

\begin{proof}
Combine Theorems \ref{RBr} and \ref{pg>0} with Lemma \ref{Cox-H1}.
\end{proof}

\section{$WR$ for complex surfaces}\label{sec:C}

In this section, we compare $WR(V_{\top})$ and $W(V)$ when
$V$ is a complex surface.  As in the previous section,
$V_\top = G\!\times Y$\!, where $Y\!=\!V(\C)$.

\begin{proposition}\label{Gx4-fold}
Let $Y$ be a connected 4-dimensional CW complex, and set $X=G\times Y$. 
Then $WR(X)\to\RGB(X)$ is an injection, and 
\[
WR(G\times Y) \cong \Z/2\oplus H^1(Y,\Z/2)\oplus{_2}H^3(Y,\Z).
\]
Thus $WR(X)\to\RGB(X)$ is an 
isomorphism if and only if the torsion in $H^3(Y,\Z)$ has exponent~2.
\end{proposition}

\begin{proof}
We saw in \cite[Ex.\,2.4(b)]{KSW} that $WR(G\times Y)$ is the 
cokernel of the `realization' map $KR(G\times Y)=KU(Y)\to KO(Y)$.  
Let $\widetilde{KU}(Y)$ denote the kernel of $\rank:KU(Y)\to\Z$.
By the Atiyah-Hirzebruch spectral sequence,
there is an exact sequence
\[
H^1(Y,\Z)\map{d_3} H^4(Y,\Z) \to \widetilde{KU}(Y) \map{c_1} H^2(Y,\Z)\to 0.
\]
Similarly, the map $(\rank,w_1):KO(Y)\to\Z\times H^1(Y,\Z/2)$ is onto,
and its kernel $SKO(Y)$ fits into an exact sequence
\[
H^1(Y,\Z/2) \map{d_3} H^4(Y,\Z)\to SKO(Y) \map{w_2} H^2(Y,\Z/2)\to 0.
\]
(The surjection $SKO(Y) \to H^2(Y,\Z/2)$ is $w_2$ by Theorem \ref{u=w_2(E)}.)

By naturality, there is a morphism between these spectral sequences
compatible with $KU^*(Y)\to KO^*(Y)$.  Since complex bundles are
oriented, they have $w_1=0$ \cite[12-A]{MS}, so the cokernel of 
$KU(Y)\to\Z\times H^1(Y,\Z/2)$ is $\Z/2\times H^1(Y,\Z/2)$.
The map $H^4(Y,\Z)\to H^4(Y,\Z)$ is a surjection, 
because it is induced from the coefficient isomorphism 
$\Z\cong KU^4\map{\simeq} KO^4$.
Hence the cokernel of $\widetilde{KU}(Y)\to SKO(Y)$ is
the cokernel of $H^2(Y,\Z)\to H^2(Y,\Z/2)$, i.e., ${_2}H^3(Y,\Z)$.

Finally, $WR(X)$ is a module over $WR(G)=\Z/2$, so all of the
potential extensions split.
\end{proof}

\begin{subremark}
As remarked in the proof of \ref{Gx4-fold},
$KO(Y)\map{w_1} H^1(Y,\Z/2)$ vanishes on the image of $KU(Y)$.
By the remark after Definition \ref{def:w_n},
this implies that the map $w_1:WR(X)\to \H_G^1(X,\Z/2)=H^1(Y,\Z/2)$ 
is well defined when $X=G\times Y$.
\end{subremark}

Now suppose that $V$ is a smooth quasi-projective complex surface.
F.\ Fern\'andez-Carmena \cite{Fernandez} proved that $W(V)$ is a split
extension of $\Z/2\times H^1_\et(V,\Z/2)$ by 
$H^0(V,\cH^2)\cong {_2}\!\Br(V)$, the quotient of $H^2_\et(V,\Z/2)$ 
by the image $\Pic(V)/2$ of the Chern class $c_1$. That is, there is a 
split extension 
\begin{equation}\label{W(complex)}
\xymatrix{
0\ar[r]&{_2}\Br(V)\ar[r]& W(V)\ \ar[r]^-{\rank,w_1}&
\ \Z/2\!\times\!H_\et^1(V,\Z/2)\ar[r]&0.
}\end{equation}

When $V$ is a projective surface, its {\it geometric genus} $p_g$
is the same as $h^{0,2}$.
Surfaces with $p_g=0$ include the projective plane $\bP^2$,
rational surfaces, ruled surfaces, K3 surfaces, and Enriques surfaces
(see \cite{Hart}).
Some surfaces of general type also have $p_g=0$, such as Godeaux surfaces,
Burniat surfaces and Mumford's fake projective plane.

\begin{theorem}\label{pg=0}
Suppose that $V$ is a smooth projective surface over $\C$. 
Then there is an exact sequence
\[
0\to (\Z/2)^{\rho} \to W(V)\to WR(V_\top) \to 0,
\]
where $\rho\ge2\,p_g$, and $\rho=0$ when $p_g=0$.
Thus $W(V)\to WR(V_\top)$ is an isomorphism if and only if ~$p_g=0$.
\end{theorem}

\begin{proof}
By Theorem \ref{h02=0},
$_2\Br(V) \cong (\Z/2)^\rho \oplus\ {_2}H^3(V(\C),\Z)$.
Now combine \eqref{W(complex)} and Proposition \ref{Gx4-fold}.
\end{proof}

\begin{remark}
Suppose that $V$ is a smooth quasi-projective surface over $\C$, 
which is not projective. Then we still know that 
the kernel and cokernel of $W(V)\to WR(V_\top)$ are the same as
the kernel and cokernel of\; ${_2}\Br(V)\to{_2}H^3(V(\C),\Z)$.
This follows from Proposition \ref{Gx4-fold} and \eqref{W(complex)}.
\end{remark}


The invariant $\rho$ is arithmetic, and not topological, 
as the following example shows.

\begin{example}\label{ExEcomplex}
Consider the surface $V=E\times E$,  where
$E$ is an elliptic curve over $\C$. 
Then $\rho\ge2$, since $p_g(V)=1$. 
By the K\"unneth formula, $X=V_\top$ has $H^1(X,\Z/2)=(\Z/2)^4$
and $H^3(S^1,\Z)$ is torsionfree. 
By \eqref{eq:RGB(V)} and Proposition \ref{Gx4-fold},
\[
WR(V_\top) \cong \RGB(V_\top)\cong \Q_2(V_\top)\cong \Z/2 \oplus (\Z/2)^4.
\]
In contrast, Theorem \ref{h02=0} and \eqref{W(complex)} show that 
\[
\BW(V)\cong \RGB(X)\oplus (\bQ/\Z)^\rho \textrm{ and } 
W(V)\cong WR(V_\top) \oplus (\Z/2)^\rho,
\]
where $\rho$ is either 4 (the general case) 
or 3 (when $E$ has complex multiplication).
This follows from Kummer theory, $H^2(X,\Z)\cong\Z^6$, and the fact that
the Picard group $\Pic(V)$ is $NS(V)\oplus\Pic^0(V)$, where 
$\Pic^0(V)$ is divisible and the N\'eron-Severi group $NS(V)$ 
is either $\Z^2$ (the general case) or $\Z^3$ 
(when $E$ has complex multiplication).
\end{example}

\goodbreak
\section{$WR$ and $\RGB$ for surfaces with no $\R$ points}
\label{sec:free}

In this section, we compare $W(V)$ and $WR(V_\top)$ when
$V$ has no $\R$-points, i.e., $G$ acts freely on $V_\top$.
Theorem \ref{pg=0} describes the situation when $V$ is defined over $\C$,
so we may assume that $V$ is geometrically connected, i.e.,
$V_\top$ is connected.

Our first goal is to determine $WR(X)$ when $G$ acts freely on $X$
and $\dim(X)\le4$.  By Propoition \ref{ext:Q_2}, $\Q_2(X)$ is a nontrivial
extension of $\Z/2$ by $\H^1_G(X,\Z/2)=H^1(X/G,\Z/2)$. If we write
$\tH^1(X/G,\Z/2)$ for the quotient of $H^1(X/G,\Z/2)$ by the class of
$\R_X(1)$, then $\Q_2(X)$ is the product of $\Z/4$ and $\tH^1(X/G,\Z/2)$.


\begin{theorem}\label{WRX-free}
Let $X$ be a connected 4-dimensional $G$-CW complex.
If $G$ acts freely on $X$, then $WR(X)$ is a $\Z/8$-algebra, and:
\begin{enumerate}
\item $WR(X)$ is an extension:
\[
0 \to {_2}\H^3_G(X,\Z(1)) \to WR(X) \to \Z/4\times\tH^1(X/G,\Z/2)\to0;
\]
\item $\Cl: WR(X) \to \RGB(X)$ is an injection.  It is an isomorphism
if and only if the torsion in $\H^3_G(X,\Z(1))$ has exponent~2;
\end{enumerate}
\end{theorem}

\begin{proof} 
The image of $WR(X)$ in $\RGB(X)$ is an
extension of $\Q_2(X)$ by ${_2}\!\RBr(X)$, by Theorem \ref{WR-RGB}. Moreover,
$\RBr(X)$ is the torsion
subgroup of $\H_G^3(X,\Z(1))$ by Theorem \ref{RBr}. Thus
it suffices to show that the kernel of $WR(X)\to\Q_2(X)$
has the same cardinality as ${_2}\!\RBr(X)$.

Recall from \cite[2.4(b)]{KSW} that $KO_G(X)=KO(X/G)$.
We will determine $WR(X)$ by comparing the Atiyah-Hirzebruch 
spectral sequences for $KO_G(X)$ and $KR(X)$,
using the hyperbolic map $h: KR\to KO_G$.

It is convenient to write $Y$ for $X/G$, so $KO_G(X)=KO^0(Y)$.
Recall that the rank and determinant define a surjection
$KO(Y)\to\Z\oplus H^1(Y,\Z/2)$, and we write $SKO(Y)$
for the kernel of this surjection. 
The Atiyah-Hirzebruch spectral sequence for $KO^*(Y)$ yields a
filtration on $KO^0(Y)$ whose associated graded groups are
$\Z$, $H^1(Y,\Z/2)$, $H^2(Y,\Z/2)$, 0, and the cokernel of
$d_3^{1,-2}:H^1(Y,\Z/2)\to H^4(Y,\Z)$. The first two layers correspond to the 
rank and determinant, and the rest describe $SKO(Y)$ as an extension.


There is also an Atiyah-Hirzebruch spectral sequence for $KR^*(X)$,
with $E_2^{p,q}=0$ when $q$ is odd, and $E_2^{p,q}=H^p(Y,\Z)$
when $q\equiv0\pmod4$.  When $q\equiv2\pmod4$, we have
$E_2^{p,q}=H^p(Y,\Z(1))$ (cohomology with twisted coefficients);
this is because $G$ acts freely, so
$X$ has a covering by open sets of the form $G\times U$.
Thus the associated graded groups of $KR^0(X)$ are
$\Z$, 0, $H^2(Y,\Z(1))$, 0, and the cokernel of
$d_3:H^1(Y,\Z/2)\to H^4(Y,\Z)$.

We now compare the two spectral sequences.
The top quotient is the injection $\Z\smap{2}\Z$, so modulo the 
second layer $I_2$ of the filtration of $WR(X)$ we have the extension
$\Q_2(X)$ of $\Z/2$ by $\H_G^1(X,\Z/2)$ discussed in
Lemmas \ref{WR/I2} and \ref{w1 on I}.

The second layer is determined by the following
commutative diagram with exact sequences; the map
$c_1$ is onto by Remark \ref{Chern}, and $w_2$ is onto because
$\H^2_G(X,\Z/2)\cong H^2(Y,\Z/2)$.
\begin{align*}\xymatrix@R=1.5em{
\H^1_G(X,\Z(1))\ar[r]^-{d_3} \ar[d] & H^4(Y,\Z) \ar[r] \ar[d]^{\cong}  
&\widetilde{KR}(X) \ar[r]^-{c_1} \ar[d]^h 
& \H^2_G(X,\Z(1))\ar[r] \ar[d] & 0, \\
\H^1_G(X,\Z/2) \ar[r]^-{d_3} & H^4(Y,\Z)\ar[r] 
&SKO(Y) \ar[r]^-{w_2}& \H^2_G(X,\Z/2)\ar[r]& 0.
}\end{align*}
\noindent
As in the proof of Proposition \ref{Gx4-fold}, the second vertical 
is the isomorphism $H^4(Y,KU^4)\!\to\!H^4(Y,KO^4)$.
Thus the cokernel of the third vertical $h$
is the cokernel ${}_2\H^3_G(X,\Z(1))$
of $\H^2_G(X,\Z(1))\!\to\!\H^2_G(X,\Z/2)$.
\end{proof}

Recall that $\RGB(X)$ is an extension of $\Q_2(X)$
by the torsion subgroup of $\H_G^3(X,\Z(1))$, and that $\Q_2(X)$
is the nontrivial extension $\Z/4\times\tH^1(X,\Z/2)$ of
$\Z/2$ by $H^1(X,\Z/2)$, $\tH^1(X,\Z/2)$ being the cokernel of
$H^1(\pt,\Z/2)\to H^1(X,\Z/2)$.
(See Theorems \ref{RBW} and \ref{RBr}, and Proposition \ref{ext:Q_2}.)
The following example shows that the extensions for $\RGB(X)$ and
$WR(X)$  can be nontrivial.

\begin{example}
Let $X=S^{5,0}$ be the 4-sphere with antipodal involution.
Then 
\[  WR(X)\cong\RGB(X)=\Z/8. \]
Indeed, $X$ is the 4-skeleton of $EG=\R\bP^\infty$, so the 
Borel cohomology $\H_G^n(X,-)$ agrees with 
$\H_G^n(EG,\!-)\cong\H_G^n(\pt,\!-)$ for $n\!<\!4$. 
In particular, 
\[
\H_G^3(X,\Z(1))\cong \H_G^3(\pt,\Z(1))\cong H^3(G,\Z(1))\cong\Z/2.
\]
\goodbreak\noindent
Hence $WR(X)\cong\RGB(X)$ by Theorem \ref{WRX-free}, and
we showed in \cite[Ex.\,2.5]{KSW} that $WR(X)\cong\Z/8$.
\end{example}




Now suppose that $V$ is a smooth geometrically connected algebraic surface 
defined over $\R$, with $V(\R)=\emptyset$.  
It is well known that $W(V)$ is a nontrivial extension of $\Z/2$ by
the augmentation ideal $I(V)$.
The following result is due to Sujatha \cite[Lem.\,3.4]{Sujatha}.

\begin{theorem}[Sujatha]\label{W(V)}
Let $V\!$ be a smooth geometrically connected surface over $\R$
with no real points.  Then we have a short exact sequence
\[
0 \to {}_2\!\Br(V)\to W(V) \to \Z/4\times\tH^1_\et(V,\Z/2) \to 0.
\]
The map from $W(V)$ to $\Q_2(V)=\Z/4\times\tH^1_\et(V,\Z/2)$
is given by the rank and discriminant of a symmetric form.
\end{theorem}

\begin{proof} 
Since $V$ has no real points, $W(V)$ is a torsion group.
Sujatha proves in \cite[2.1, 2.2]{Sujatha} that
the discriminant $I(V)\to H^1_\et(V,\Z/2)$ 
is a surjection with kernel $I_2(V)$, and that the Hasse invariant
$I_2(V)\to {_2}\!\Br(V)$ is an isomorphism.
Sujatha writes $\Gamma_t(V,\cH^1)$ and $\Gamma_t(V,\cH^2)$ for 
$H^1_\et(V,\Z/2)$ and ${_2}\!\Br(V)$; the reinterpretation 
is standard, using \cite{BlochOgus}.
\end{proof}

Using Theorem \ref{BW}, we conclude:

\begin{subcor}\label{W-BW-Q}
For $V$ as in Theorem \ref{W(V)}, the map
$W(V)\to\BW(V)$
is an injection, and is an isomorphism exactly when
the Brauer group $\Br(V)$ has exponent~2.
\end{subcor}


Putting together the pieces, and using Lemma \ref{Cox-H1}, we see that
in the setting of Theorem \ref{WRX-free},
we have a commutative diagram with exact rows.
\begin{align*}\xymatrix@R=1.5em{
0\ar[r]& {_2}\Br(V) \ar[r]\ar[d] &W(V) \ar[r]\ar[d] 
        & \Q_2(V)\ar[d]^{\cong}\ar[r] &0 \\
0\ar[r]& {_2}\RBr(V_\top) \ar[r] &WR(V_\top) \ar[r]
        & \Q_2(V_\top)\ar[r] &0
}\end{align*}

Recall that the Lefschetz number $\rho_0$ is the rank of the cokernel
of the equivariant Chern class $\Pic(V)\to H^2(X,\Z(1))$ of
Remark \ref{Chern}, and that $p_g=h^{0,2}$.
Combining Theorems \ref{pg>0} and \ref{real:pg=0}, 
we have proven:

\begin{theorem}\label{WR:pg=0}
Let $V$ be a smooth geometrically connected projective surface over $\R$
with no real points. Then there is a split exact sequence
\[
0\to (\Z/2)^{\rho_0} \to W(V)\to WR(V_\top) \to 0 \\
\]
(with $\rho_0$ as in Theorem \ref{real:pg=0}),
and  $W(V)\to WR(V_\top)$ is
an isomorphism if and only if $V$ has geometric genus~$p_g=h^{0,2}=0$.
\end{theorem}

\newpage 
\appendix

\section{An equivariant Theorem of Brown}

In this Appendix, we give an equivariant version of 
a Theorem of E.\ Brown, taken from \cite[7.1]{Dold}.
Let $W_*$ denote the category  of finite pointed, connected CW complexes.
A homotopy invariant contravariant functor $F:W_*\to\mathbf{Sets}$
is called {\it half-exact} if the natural maps
$F(X_1\vee X_2)\to \prod F(X_i)$ are bijections and for every
subcomplex $A$ of $X$ the map $F(X\cup_A X)\to F(X)\times_{F(A)}F(X)$ is onto.

Brown's Theorem: 
Let $\rho:E\to F$ be a natural transformation between half-exact 
functors from the category $W_*$ to $\mathbf{Sets}$.
If $\rho_{S^n}$ is a bijection for every sphere $S^n$, then
$\rho_X$ is a bijection for every connected $X$ in $W_*$.


Let $G$ be a finite group, and let $W_*(G)$ denote the category of
finite pointed  $G$-spaces with $X/G$ connected, and equivariant maps. 
The notion of a $G$-half-exact functor is the same as that of a half-exact
functor, with $W_*$ replaced by $W_*(G)$.
Given an equivariant map $f: A\to X$ in $W_*(G)$, we can form the 
mapping cone $C_f$ and the long exact Puppe sequence of pointed sets
\begin{equation*}
\cdots F(SX) \map{Sf^*} F(SA) \to F(C_f) \to F(X) \map{f^*} F(A).
\end{equation*}
See \cite[III.4]{Bredon}.  Because $SA$ is a co-$H$-space object
in $W_*(G)$, $F(SA)$ is a group, and acts on $F(C_f)$; the action
$F(SA)\times F(C_f)\to C_f$ is given by the 
usual co-multiplication $\delta:C_f\to SA\vee C_f$.
See \cite[III(6.20)]{Whitehead}.
The proof of \cite[III(6.21)]{Whitehead} and/or
\cite[5.6]{Dold} goes through to prove:

\begin{lemma}\label{action}
Elements $a,b\in F(C_f)$ agree in $F(X)$ if and only if 
there is a $\gamma\in F(SA)$ so that $\gamma\cdot a = b$.
\end{lemma}

Recall that for every orbit $G/H$ of $G$ there is a family of
test $G$-spaces $e^n_H=S^n\wedge (G/H)_+$.  If $X^{(n)}$ is the $n$-skeleton
of $X$, its $(n+1)$-skeleton is the mapping cylinder of attaching maps
$\bigvee_i e_i\to X^{(n)}$, where each $e_i$ is $S^n\wedge(G/H_i)_+$
for some $H_i$.

\begin{theorem}\label{thm:Brown}
Let $\rho:E\to F$ be a natural transformation between half-exact functors
from $W_*(G)$ to Sets such that $\rho_{X}$ is a bijection for
every $G$-cell $X=S^n(G/H_+)$, and for every $X$ of dimension $\le1$.
Then $\rho_X: E(X)\to F(X)$ is a bijection for every connected
$X$ in $W_*(G)$.
\end{theorem}

\begin{proof}
The result is true for 1--dimensional $X$, 
by assumption.  
We proceed by induction on $\dim(X)$.
Suppose that $X=X^{(n+1)}$ and the result is true for the $n$-skeleton 
$X^{(n)}$. Then $X$ is the mapping cylinder of the attaching maps
$\bigvee_i e_i \map{\alpha} X^{(n)}$.
By hypothesis, $F(\bigvee_i e_i)\cong\prod F(e_i)$, 
so we have a diagram of pointed sets with exact rows:
\[\xymatrix@R=2.0em{ 
E(SX^{(n)}) \ar[r]^{S\alpha^*}\ar[d]^{\cong} & \prod E(Se_i)) 
\ar[r]\ar[d]^{\cong} & E(X) \ar[r]^{i^*}\ar[d]^{\rho_X} &
E(X^{(n)}) \ar[r]^{\alpha^*}\ar[d]^{\cong} &\prod E(e_i)\ar[d]^{\cong} 
\\
F(SX^{(n)}) \ar[r]^{S\alpha^*} & \prod F(Se_i) \ar[r] &
F(X) \ar[r]^{i^*} & F(X^{(n)}) \ar[r]^{\alpha^*} & \prod F(e_i).
}\]
For notational convenience, we write $C$ for $\bigvee Se_i=S(\bigvee e_i)$.

To see that $\rho_X$ is onto, suppose given $b\in F(X)$; a 
diagram chase shows that 
there is an $a\in E(X)$ so that $\rho_X(a)$ and $b$ agree
in $F(X^{(n)})$. By Lemma \ref{action}, there is a $\gamma$ in
$E(C)\cong F(C)$ so that
$b=\gamma\cdot \rho_X(a) = \rho_X(\gamma\cdot a)$.

To see that $\rho_X$ is into, we mimick Dold's argument
\cite[7.2]{Dold}, using the mapping torus $W$ of 
$\id,\delta: C\rightrightarrows C\vee X$.  
Given $a,a_1$ in $E(X)$ agreeing in $F(X)$, 
a diagram chase (using Lemma \ref{action}) shows that 
$a_1=\gamma\cdot a$ for some $\gamma\in E(C)$, and hence
$\gamma$ stabilizes $b=\rho_X(a)$.
By \cite[5.7]{Dold}, the pair $(\gamma,b)$ lifts to $F(W)$.
Since $\rho_W$ is onto (by Lemma \ref{action}), 
$(\gamma,b)$ lifts to $w\in E(W)$. The image $(\gamma',a')$ 
of this element in $E(C)\times E(X)$ satisfies 
$\gamma'\cdot a'=a'$, and maps to $(\gamma,b)$ in
$E(C)\times F(X)$, so $\gamma'=\gamma$ 
(recall that $E(C)\cong F(C)$ by assumption).
In addition, since $i^*(a)=i^*(a')$ (by the same argument),
$a = \gamma''\cdot a'$ for some $\gamma''$. Now
$E(C)$ is a commutative group, because $n\ge2$, so we conclude:
\[
a_1=\gamma\cdot a = \gamma\cdot(\gamma'' \cdot a') 
   = \gamma''\cdot(\gamma\cdot a')
   = \gamma''\cdot a' = a.   \qedhere
\]
\end{proof}


\end{document}